\documentclass[a4paper,12pt]{article}
 \usepackage[utf8]{inputenc}
 \usepackage{amsmath}
 \usepackage{amssymb}
 \usepackage{amsthm}
 \usepackage{vmargin} 
 \usepackage{graphicx}
 \usepackage{color}
 \usepackage{hyperref}
 \usepackage{cite}

\usepackage{unicodeintex}

\renewcommand{\d}{\mathrm d}

\newtheorem{definition}{Definition}[section]

\newtheorem{theorem}{Theorem}
\newtheorem{proposition}{Proposition}[section]

\newtheorem{lemma}{Lemma}

\title{Body-attitude alignment: first order phase transition, link with rodlike polymers through quaternions, and stability}
\author{Amic Frouvelle\thanks{CEREMADE, CNRS, Université Paris-Dauphine, Université PSL, 75016 Paris, France. E-mail: \texttt{frouvelle@ceremade.dauphine.fr}} \thanks{CNRS, Université de Poitiers, UMR 7348 --- Laboratoire de Mathématiques et Applications (LMA), 86000 Poitiers, France. E-mail: \texttt{amic.frouvelle@math.univ-poitiers.fr}}}
\date{}

\begin{document}
\maketitle

\begin{abstract}
  We present a simple model of alignment of a large number of rigid bodies (modeled by rotation matrices) subject to internal rotational noise. The numerical simulations exhibit a phenomenon of first order phase transition with respect the alignment intensity, with abrupt transition at two thresholds. Below the first threshold, the system is disordered in large time: the rotation matrices are uniformly distributed. Above the second threshold, the long time behaviour of the system is to concentrate around a given rotation matrix. When the intensity is between the two thresholds, both situations may occur.

  We then study the mean-field limit of this model, as the number of particles tends to infinity, which takes the form of a nonlinear Fokker--Planck equation. We describe the complete classification of the steady states of this equation, which fits with numerical experiments. This classification was obtained in a previous work by Degond, Diez, Merino-Aceituno and the author, thanks to the link between this model and a four-dimensional generalization of the Doi--Onsager equation for suspensions of rodlike polymers interacting through Maier--Saupe potential.

  This previous study concerned a similar equation of BGK type for which the steady-states were the same. We take advantage of the stability results obtained in this framework, and are able to prove the exponential stability of two families of steady-states: the disordered uniform distribution when the intensity of alignment is less than the second threshold, and a family of non-isotropic steady states (one for each possible rotation matrix, concentrated around it), when the intensity is greater than the first threshold. We also show that the other families of steady-states are unstable, in agreement with the numerical observations.
\end{abstract}

\section*{Introduction}

The mathematical study of active matter, such as aligning self-propelled particles, is now a well established field of research, inspired for instance by phase transition phenomena that appear in the Vicsek model~\cite{vicsek1995novel,chate2008collective}. Following the kinetic approach introduced in~\cite{degond2008continuum}, a simple model of alignment of unit vectors subject to internal rotational noise gives rise to a continuous phase transition at the kinetic level~\cite{frouvelle2012dynamics}. When the alignment intensity (that we call~$ρ$, since it is related to the local density~$ρ$ of particles in the inhomogeneous version~\cite{degond2013macroscopic}, where the unit vectors represent the velocities of self-propelled particles) is below a threshold~$ρ_c$, the only stable steady-state is the uniform distribution on the unit sphere. On the other hand, when~$ρ>ρ_c$, this isotropic equilibria becomes unstable and a family of stable equilibria arises: von Mises distributions with concentration parameter depending on~$ρ$, around a given unit vector. When setting the intensity of alignment as a nonlinear function of the order parameter of the system~\cite{degond2015phase}, this continuous phase transition may become a discontinuous one (or first order), with hysteresis phenomenon: a second threshold~$ρ^*<ρ_c$ appears, the uniform equilibrium distribution being stable for~$ρ<ρ_c$ and the concentrated distributions being stable for~$ρ>ρ_*$. Around those thresholds, the order parameter cannot vary continuously from a family of equilibria to the other.

Recently, in a work with Degond and Merino-Aceituno~\cite{degond2017new} we extended the model of self-propelled particles of Degond and Motsch~\cite{degond2008continuum} to the case where the orientation of particles are not only given by their velocity (a unit vector) but by their whole body attitude (an orthonormal frame, given by a rotation matrix). Then, still with Degond and Merino-Aceituno, together with Trescases~\cite{degond2018quaternions} we proposed a similar model based on quaternion representation for rotation matrices, and the models appeared to be equivalent. In these models, the interaction was normalized and no phenomenon of phase transition could occur, but we remarked that the non-normalized version may lead to such a phenomena. Finally, with Degond, Diez and Merino-Aceituno~\cite{degond2020phase} we managed to treat this phenomenon of phase transition in a homogeneous Bhatnagar--Gross--Krook (BGK) model, thanks to this link with unit quaternions and an analogy with a four-dimensional generalization of the Doi--Onsager equation for suspensions of rodlike polymers interacting through Maier--Saupe potential. Indeed, the compatibility equation we need to solve to determine the possible steady-states can be reformulated in this quaternionic formulation, and leads to a compatibility equation for the Maier--Saupe potential in dimension~$4$, which was solved in~\cite{wang2008unified}. We obtain a discontinuous phase transition with two thresholds~$ρ^*<ρ_c$, still with the same two types of stable equilibria: the uniform distribution for~$ρ<ρ_c$, and a family of generalized von Mises distributions, concentrated around a given rotation matrix when~$ρ>ρ^*$.

The aim of this paper is twofold. We first want to introduce the model of alignment of rigid bodies through numerical simulations of the particle system, in order to present the first order phase transition that we observe numerically. And then we want to provide a rigorous mathematical description of this phase transition phenomenon, at the kinetic level: the mean-field limit of the particle system when the number of particles is large is given by a nonlinear Fokker--Planck equation, for which the steady states are the same as those characterized in~\cite{degond2020phase} for the BGK equation. The main result of this article is that we have a fine description of the long-time behaviour of the solution to the Fokker--Planck equation: we classify all the families of equilibria regarding their stability, and prove the exponential stability of the uniform equilibrium when~$ρ<ρ_c$ and of the concentrated von Mises distributions when~$ρ>ρ^*$.

In Section~\ref{sec-numerical}, we present the framework of our model: a system of coupled stochastic differential equations for~$N$ matrices in~$SO_3(ℝ)$. We present a time discretization scheme of Euler--Naruyama type, and provide numerical simulations which illustrate the phenomenon of first order phase transition. In Section~\ref{sec-meanfield}, we describe the mean-field limit of this system, which takes the form of a nonlinear Fokker--Planck equation. We give general results on the behaviour of the solution of this evolution equation, and we show that the determination of its steady states amounts to solve a matrix compatibility equation. Thanks to the free energy associated to the Fokker--Planck equation, the uniform equilibria is shown to be unstable for~$ρ>ρ_c=6$, proving that in that case there are others solutions than~$0$ for the compatibility equation. Section~\ref{sec-polymers} is a summary of the results of~\cite{degond2020phase} to solve this compatibility equation: we present the link between rotation matrices and unit quaternions, and the fact that the compatibility equation can be transformed to a compatibility equation for~$Q$-tensors which was solved in~\cite{wang2008unified}. We therefore get a precise description of all the steady-states of the equation, and a way to obtain the second threshold~$ρ^*$ (as the minimum of a one-dimensional function) such that for~$ρ⩾ρ^*$ there exists non-trivial steady-states. In Section~\ref{sec-BGK}, we summarize the results of~\cite{degond2020phase} regarding the stability of these equilibria in the framework a BGK equation (which shares the same steady-states), and we are able to use these results to obtain the classification of the steady-states, as critical points of the free energy. In particular we show that three families of equilibria are unstable, and the remaining two other types are local minimizers of the free energy: the uniform distribution when~$ρ<ρ_c$ and the concentrated von Mises distributions for~$ρ>ρ^*$. Finally, Section~\ref{sec-exponential} is devoted to the main new result of this paper: the exponential stability of these two types of steady-states. In Theorem~\ref{thm-expstability-FP}, we prove that if a function~$f_0$ is sufficiently close to the set of equilibria (in relative entropy), then there exist such an equilibrium~$f_∞$ such that the solution of the Fokker--Planck equation converges exponentially fast towards~$f_∞$ (still in relative entropy). We finish this last section by some comments and perspectives.

\section{Numerical evidence of a  first-order phase transition in a system of interacting particles}
\label{sec-numerical}
\subsection{A simple SDE on~$SO_3(ℝ)$ and its time-discretization}
First of all let us recall some basic facts about~$SO_3(ℝ)$.

\begin{definition}
  \label{crossmatrix}
  For any~$\mathbf{u}={\scriptstyle\begin{pmatrix}u_1\\u_2\\u_3\end{pmatrix}}∈ℝ^3$ we denote by~$[\mathbf{u}]_×={\scriptstyle\begin{pmatrix}0&-u_3&u_2\\u_3&0&-u_1\\-u_2&u_1&0\end{pmatrix}}$ the (antisymmetric) matrix associated to the linear map~$\mathbf{v}∈ℝ^3↦\mathbf{u}\times\mathbf{v}$ in the canonical basis. 
\end{definition}

\begin{proposition}[Rodrigues' formula]
 
  Any special orthogonal matrix~$A∈SO_3(ℝ)$ can be written as a rotation around an axis in~$ℝ^3$. More precisely, there exists a unique angle~$θ∈[0,π]$ and a unit vector~$\mathbf{n}∈𝕊_2$ such that~$A$ is the rotation~$R(θ,\mathbf{n})$ of angle~$θ$ around the axis directed by~$\mathbf{n}$, given by the following formula:
  \begin{equation}
    \label{eqRodrigues}
    R(θ,\mathbf{n})=\exp(θ[\mathbf{n}]_×)=\cosθ \,I_3+\sinθ[\mathbf{n}]_×+(1-\cosθ)\mathbf{n}\mathbf{n}^{\top}.
  \end{equation}
  where~$I_3$ is the identity matrix. When~$θ∈(0,π)$, the unit vector~$\mathbf{n}$ is unique. When~$θ=π$ there are two such vectors~$\mathbf{n}$, opposite one to the other. And when~$θ=0$, any unit vector~$\mathbf{n}$ can be used.
\end{proposition}

To introduce the model and some important notations, we first start with a simple stochastic differential equation (SDE) modeling a rotation matrix~$A(t)∈SO_3(ℝ)$ trying to align with another fixed rotation matrix~$A_0∈SO_3(ℝ)$, with strength of alignment~$ν>0$, and subject to angular noise of intensity~$τ>0$ :
\begin{equation}
  \label{simpleSDE}
  \d A=-ν∇_A(\tfrac12∥A-A_0∥^2) \d t + 2 \sqrt{τ} P_{T_{A}}\circ \d B_t.
\end{equation}

To give a meaning to the previous equation, let us describe the terms one by one, from left to right. We need to define a metric on~$SO_3(ℝ)$ in order to define the gradient~$∇_A$. As it is usually the case in~$SO_3(ℝ)$, we will take the metric induced by the scalar product in~$M_3(ℝ)$ given by
\begin{equation}\label{dotSO3}
  A·B=\frac12\mathrm{Tr}(AB^{\top}).
\end{equation}
One of the reasons to take this metric is that the geodesic distance between a matrix~$A∈SO_3(ℝ)$ and its composition by a rotation matrix of angle~$θ∈[0,π]$ is exactly~$θ$. Said differently, if~$\mathbf{n}∈𝕊_2$, then the curve~$θ∈ℝ↦R(θ,\mathbf{n})A$ given by the formula~\eqref{eqRodrigues} is a geodesic travelled at unit speed. The other reason is that the map~$\mathbf{u}∈ℝ^3↦[\mathbf{u}]_×$ given by Definition~\ref{crossmatrix} is an isometry from~$ℝ^3$ to the antisymmetric matrices (which is the Lie algebra of~$SO_3(ℝ)$). The norm~$∥A-A_0∥$ in the SDE~\eqref{simpleSDE} is the one associated to this scalar product. The operator~$P_{T_{A}}$ is the orthogonal projection on the tangent space of~$SO_3(ℝ)$ at~$A$, given by~$P_{T_A}H=\frac12(H-AH^{\top}A)$. The notation~$\circ$ in the SDE~\eqref{simpleSDE} means that it must be understood in the Stratonovich sense, and the Brownian motion~$B_t$ is a~$3×3$ matrix whose entries are independent real standard Brownian motions\footnote{Note that this does not give a standard Brownian motion on the euclidean space~$M_3(ℝ)$, equipped with this scalar product, but~$\widetilde{B}_t=\sqrt{2}B_t$ is such a standard Brownian motion. The SDE for a standard Brownian motion on the manifold, with generator~$\frac12Δ_A$, would be~$\d A=P_{T_A}\circ \d \widetilde B_t$, which explain the choice of~$2\sqrt{τ}$ instead of the usual~$\sqrt{2τ}$ in the SDE~\eqref{simpleSDE} so that the Fokker--Planck equation~\eqref{FPlinear} has the simplest coefficients.}. This ensures that the matrix~$A$ stays on~$SO_3(ℝ)$ for all time, and this is the usual way of defining SDEs on manifolds (we refer to~\cite{hsu2002stochastic} for a reference on this topic). Therefore the first term in the right-hand side of~\eqref{simpleSDE} may be written~$ν∇_A(A·A_0)$ since~$∥A∥^2=\frac32$ whenever~$A∈SO_3(ℝ)$. Finally, the law~$t↦μ(t,·)$ (with values in~$\mathcal{P}(SO_3(ℝ))$, the set of probability measures on~$SO_3(ℝ)$) of such a process satisfies the following Fokker--Planck equation:
\begin{equation}
  \label{FPlinear}
  ∂_tμ+ν∇_A·(∇_A(A·A_0)μ)=τΔ_Aμ,
\end{equation}
where~$∇_A·$ and~$Δ_A$ are the divergence and Laplace-Beltrami operators on~$SO_3(ℝ)$. Up to a time rescaling, we see that the important parameter is~$κ=\frac{ν}{τ}$, and we can then without loss of generality study the following PDE, obtained by replacing~$τ$ by~$1$ and~$ν$ by~$κ$ in~\eqref{FPlinear}:
\begin{equation}
  \label{FPlinearReduced}
  ∂_tμ=-κ∇_A·(∇_A(A·A_0)μ)+Δ_Aμ=∇_A·\Big[\exp(κ\,A·A_0)∇_A\Big(\frac{μ}{\exp(κ\, A·A_0)}\Big)\Big].
\end{equation}
In view of the above formulation, we now define the generalized von Mises distribution (a probability measure) on~$SO_3(ℝ)$ of parameter~$J∈M_3(ℝ)$ by
\begin{equation}
  \label{defVM}
  M_J(A)=\frac1{\mathcal{Z}(J)}\exp(J·A)\text{, where }\mathcal{Z}(J)=∫_{SO_3(ℝ)}\exp(J·A)\d A,
\end{equation}
the normalized volume form on~$SO_3(ℝ)$ being its Haar probability measure (this comes from invariance of the metric with respect to left or right multiplication by a given rotation matrix). Therefore it is for instance easy to see that~$\mathcal{Z}(κA_0)$ only depends on~$κ$ when~$A_0∈SO_3(ℝ)$. With this notation, we can multiply the PDE~\eqref{FPlinearReduced} by~$\frac{μ}{M_{κA_0}}$, integrate by parts and take advantage of the fact that the integral of~$μ$ on~$SO_3(ℝ)$ remains constant in time, to obtain
\begin{equation}
  \frac12\frac{\d}{\d t}∫_{SO_3(ℝ)}|\tfrac{μ}{M_{κA_0}}-1|^2M_{κA_0}\d A=-∫_{SO_3(ℝ)}∥∇_A(\tfrac{μ}{M_{κA_0}}-1)∥^2M_{κA_0}\d A.
\end{equation}
Together with a weighted Poincaré inequality on~$SO_3(ℝ)$, this shows that the solution to the PDE~\eqref{FPlinearReduced} converges exponentially fast to the von Mises distribution~$M_{κA_0}$. Let us remark that when~$κ$ is small (strong noise, or weak alignment), this distribution tends to be uniform on~$SO_3(ℝ)$, and when~$κ$ is large (strong alignment or low level of noise), it is concentrated around the maximizer of~$A↦A·A_0$, which is exactly~$A_0$, as expected.

Let us finish this subsection by describing a numerical discretization of the SDE~\eqref{simpleSDE}. By using the fact that~$∇_A(A·A_0)=P_{T_A}A_0$, and denoting by~$Π$ the orthogonal projection on~$SO_3(ℝ)$ (well-defined in a neighborhood of the manifold), a naive projected Euler-Naruyama scheme would read as follows:
\begin{equation}
  \label{naiveScheme}
  A(t+Δt)≈Π(A(t)+ν Δt\, P_{T_{A(t)}}A_0+\sqrt{Δt}\,2\sqrt{τ}P_{T_{A(t)}}{\mathcal{N}_{9}}),
\end{equation}
where~$\mathcal{N}_9$ is a three by three matrix whose~$9$ entries are independent samples of standard Gaussian distribution. One could even remove the projections on the tangent plane and use this model, easy to describe as a starting point : “Start from~$A∈SO_3(ℝ)$, move with step~$νΔt$ in the direction of the target~$A_0$, add some noise of intensity~$2\sqrt{τΔt}$ and project the result back on~$SO_3(ℝ)$”. However, there is a way to avoid sampling~$9$ entries per step and to take advantage of the Lie group structure of~$SO_3(ℝ)$ instead of computing the projection on~$SO_3(ℝ)$ (which is the polar decomposition of matrices and may have some cost). Indeed, the right-hand side of the scheme~\eqref{naiveScheme} can be written
\begin{equation*}
  Π(I_3+\tfrac12 νΔt\, [A_0A(t)^{\top}-A(t)A_0^{\top}] + \sqrt{τΔt}[{\mathcal{N}_9}A(t)^{\top}-A(t){\mathcal{N}_9}^{\top}])A(t).
\end{equation*}
Since a rotation of a standard Gaussian vector is still a standard Gaussian vector, one can see that the matrix~${\mathcal{N}_9}A(t)^{\top}$ is also a matrix whose~$9$ entries are independent samples of standard Gaussian distribution. Therefore~${\mathcal{N}_9}A(t)^{\top}-A(t){\mathcal{N}_9}^{\top}$ is an antisymmetric matrix whose independent entries are samples of centered Gaussian distribution of variance~$2$. It is then a matrix of the form~$\sqrt{2}[\boldsymbol{η}]_×$ (see Definition~\ref{crossmatrix}), where~$\boldsymbol{η}$ is a standard Gaussian vector in~$ℝ^3$. When~$H$ is a small antisymmetric matrix, a consistent approximation to~$Π(I_3+H)$ is given by~$\exp(H)$ and can be computed thanks to Rodrigues’ formula~\eqref{eqRodrigues}. Therefore a numerical scheme consistent with the naive scheme~\eqref{naiveScheme} is given by
\begin{equation}
  \label{LieScheme}
  A(t+Δt)≈\exp(\tfrac12 νΔt\, [A_0A(t)^{\top}-A(t)A_0^{\top}]+\sqrt{2τΔt}[\boldsymbol{η}]_×)A(t),
\end{equation}
where~$\boldsymbol{η}$ is a standard Gaussian vector in~$ℝ^3$.

\subsection{A system of SDEs and its numerical simulations}

We are now ready to introduce our model. In the article~\cite{degond2017new}, we considered~$N$ individuals located at positions~$X_i∈ℝ^3$ for~$1⩽i⩽N$ and with body orientations~$A_i∈SO_3(ℝ)$, moving at unit speed in the direction of their first vector~$A_i\mathbf{e}_1$ and aligning their orientations with their neighbours, as in the simple SDE~\eqref{simpleSDE}. This could take the following form\footnote{Actually, the model studied in~\cite{degond2017new} (which does not present the phenomenon of phase transition we are studying here) is a little bit more involved: each particle first chose an average target and aligns with it, instead of averaging the “forces of alignment” as it is the case in the system of SDEs~\eqref{SDEsystem-space}.}:
\begin{equation}
  \label{SDEsystem-space}
  \begin{cases}\d X_k=A_k\mathbf{e}_1 \d t\\
      \d A_k=-\sum\limits_{j=1}^Nν_{j,k}∇_{A_k}(\frac12∥A_k-A_j∥^2) \d t + 2 \sqrt{τ} P_{T_{A_k}}\circ \d B_{t,k},
\end{cases}
\end{equation}
where~$ν_{j,k}$ is the intensity at which particle~$k$ aligns with particle~$j$, and which may depend for instance on the distance~$∥X_j-X_k∥$ between the particles. We consider here a much simpler model, homogeneous in space, so we only look at~$N$ rotation matrices~$(A_i)_{1⩽i⩽n}∈SO_3(ℝ)$, with the same intensity~$\frac{ρ}N$ of alignment between any pair of particles. We are therefore interested in the following system of SDEs, using the fact that~$∇_A(\frac12∥A-A_0∥^2)=-∇_A(A·A_0)=-P_{T_A}A_0$ :

\begin{equation*}
  ∀k∈1…N, \quad \d A_k=\frac{ρ}{N}\sum\limits_{j=1}^NP_{T_{A_k}}A_j \d t + 2 \sqrt{τ} P_{T_{A_k}}\circ \d B_{t,k}.
\end{equation*}

In this model, when all the rotation matrices are close to a given one~$A_0$, the behaviour of the system can be expected to be similar to the one of the simple SDE~\eqref{simpleSDE}, and we may expect the matrices to concentrate if the alignment intensity~$ρ$ is high (or~$τ$ is low). Conversely, if they are not concentrated around some target, the average of the alignment forces is small and the noise level may prevent the matrices to align if~$ρ$ is low (or~$τ$ is high). From now on, up to rescaling time (and dividing~$ρ$ by~$τ$), we consider the case~$τ=1$ and we denote by~$J$ the average “flux”, so our system has the following form:

\begin{equation}
  \label{SDEsystem}
  \begin{cases}
    \d A_k=P_{T_{A_k}}J \d t + 2 P_{T_{A_k}}\circ \d B_{t,k},\quad (1⩽k⩽N)\\
    J(t)=\frac{ρ}{N}\sum\limits_{j=1}^NA_j(t).
\end{cases}
\end{equation}

We are then interested in the different behaviours of the system~\eqref{SDEsystem} for different values of~$ρ$. One way to measure how much matrices are concentrated is to compute the variance~$⟨∥A-⟨A⟩∥^2⟩$ (where we denote~$⟨h(A)⟩=\frac1N\sum_{j=1}^Nh(A_j)$ for any function~$h$). This nonnegative quantity is equal to~$⟨∥A∥^2⟩-∥⟨A⟩∥^2=\frac32-∥\frac{J}{ρ}∥^2$, which implies that if we define the order parameter~$c(t)$ by
\begin{equation}
  \label{def-c}
  c(t)=\frac{\sqrt2}{\sqrt3ρ}∥J(t)∥,
\end{equation}
we obtain a quantity between~$0$ (when the variance is maximal) and~$1$ (the variance is~$0$, all matrices are the same). To give a numerical illustration of the phenomenon we are interested in, we use a scheme similar to the scheme~\eqref{LieScheme} of the previous subsection: we take~$N$ matrices~$A_k∈SO_3(ℝ)$ for~$1⩽k⩽N$), a time step~$Δt$, and at each time iteration, we compute~$J=\frac{ρ}{N}\sum_{j=1}^NA_j$ and we update each~$A_k$ for~$1⩽k⩽N$ with the matrix
\begin{equation*}
  \exp(\tfrac12 Δt\, [J A_k^{\top}- A_k J^{\top}]+\sqrt{2Δt}[\boldsymbol{η}_k]_×)A_k,
\end{equation*}
where~$(\boldsymbol{η}_k)_{1⩽k⩽N}$ are independent samples of a standard Gaussian vector in~$ℝ^3$.

Figure~\ref{fig-two-examples} depicts the time evolution of the order parameter~$c(t)$ given by the formula~\eqref{def-c} for two realisations of this numerical scheme. In both cases the number of particles is~$N=500$, the time step is~$Δt=0.04$ and we run the simulation for~$100$ time iterations. In the top-left part of Figure~\ref{fig-two-examples} where we took~$ρ=1$, even if we started with all the particles in the same position (order parameter equal to~$1$), as time evolves, the order parameter becomes very small. In the top-right part, with~$ρ=10$, even if the particles were uniformly sampled on~$SO_3(ℝ)$ (order parameter close to~$0$), as time evolves, the order parameter stabilizes around a quite high value, indicating that the matrices are concentrated around a given rotation matrix. This indicates that a phase transition phenomenon is occurring with respect to the parameter~$ρ$. However, for some intermediate values of~$ρ$, as in the bottom part of Figure~\ref{fig-two-examples} where~$ρ=5$, two different behaviours may happen: starting with concentrated particles lead to an order parameter stabilizing around a non-zero value, while the configuration starting with particles uniformly sampled on~$SO_3(ℝ)$ stays with an order parameter close to~$0$ as time evolves.

\begin{figure}[h!]
  \begin{center}
\includegraphics[width=7cm]{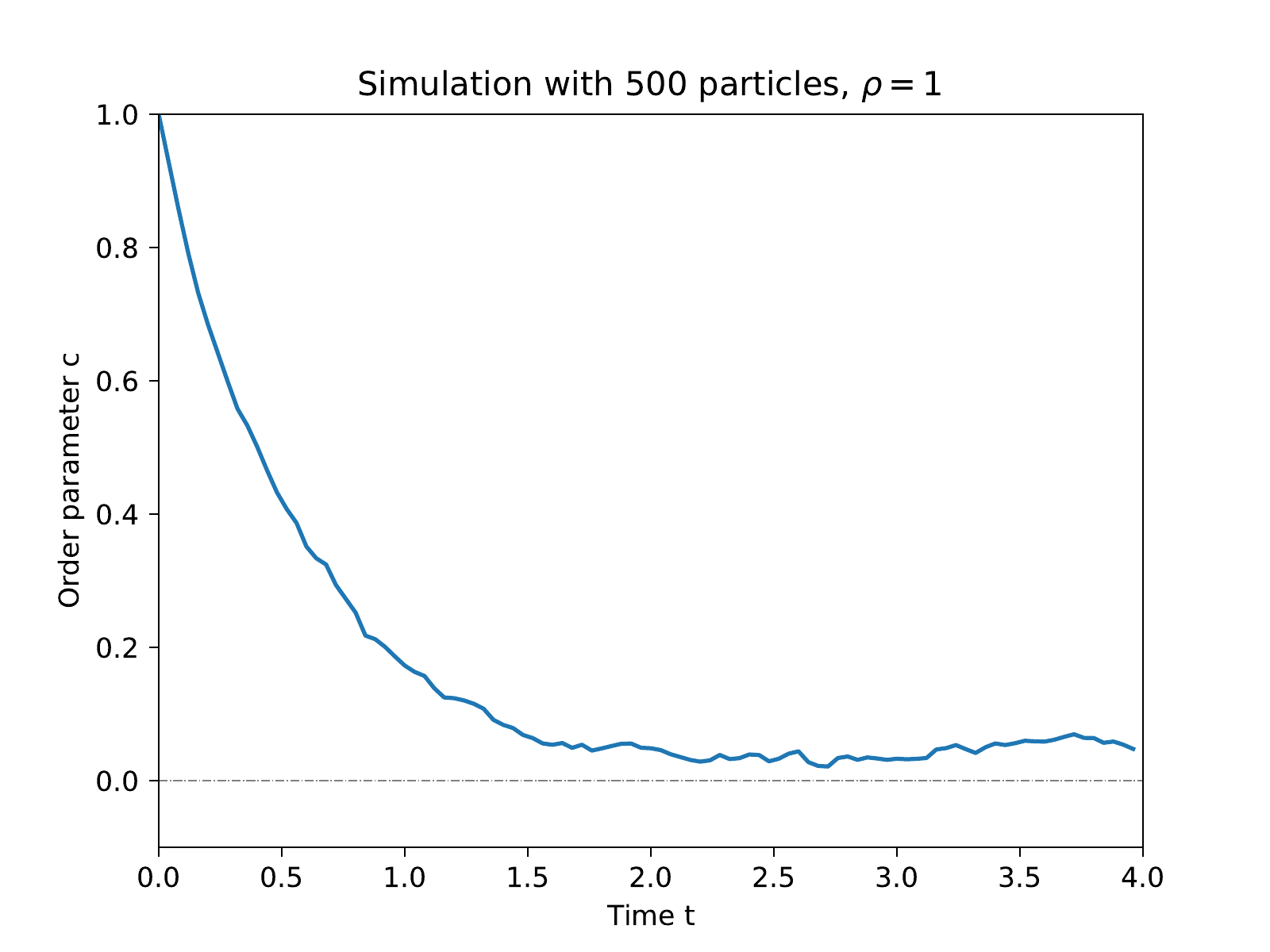}
\includegraphics[width=7cm]{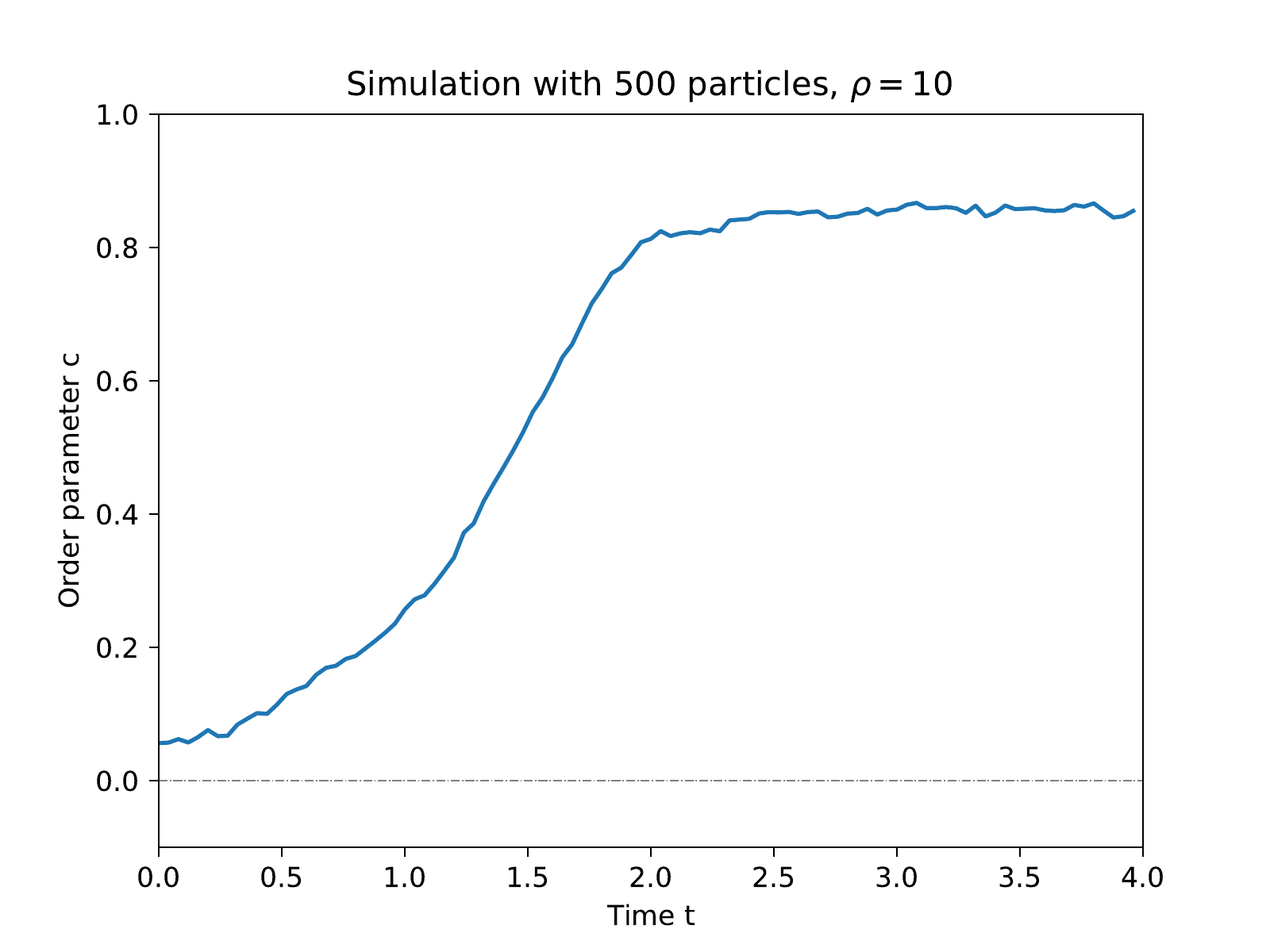}

\includegraphics[width=7cm]{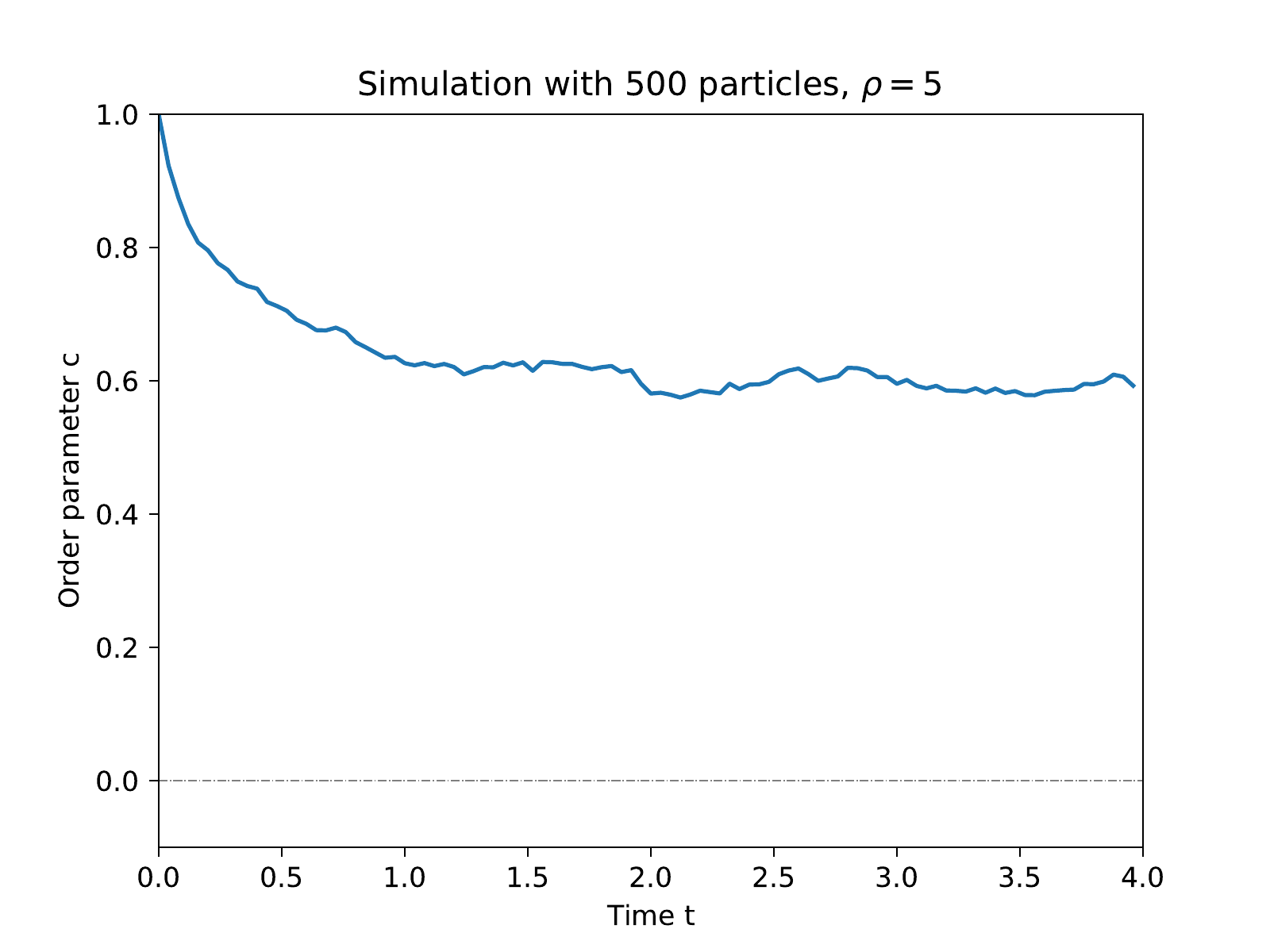}
\includegraphics[width=7cm]{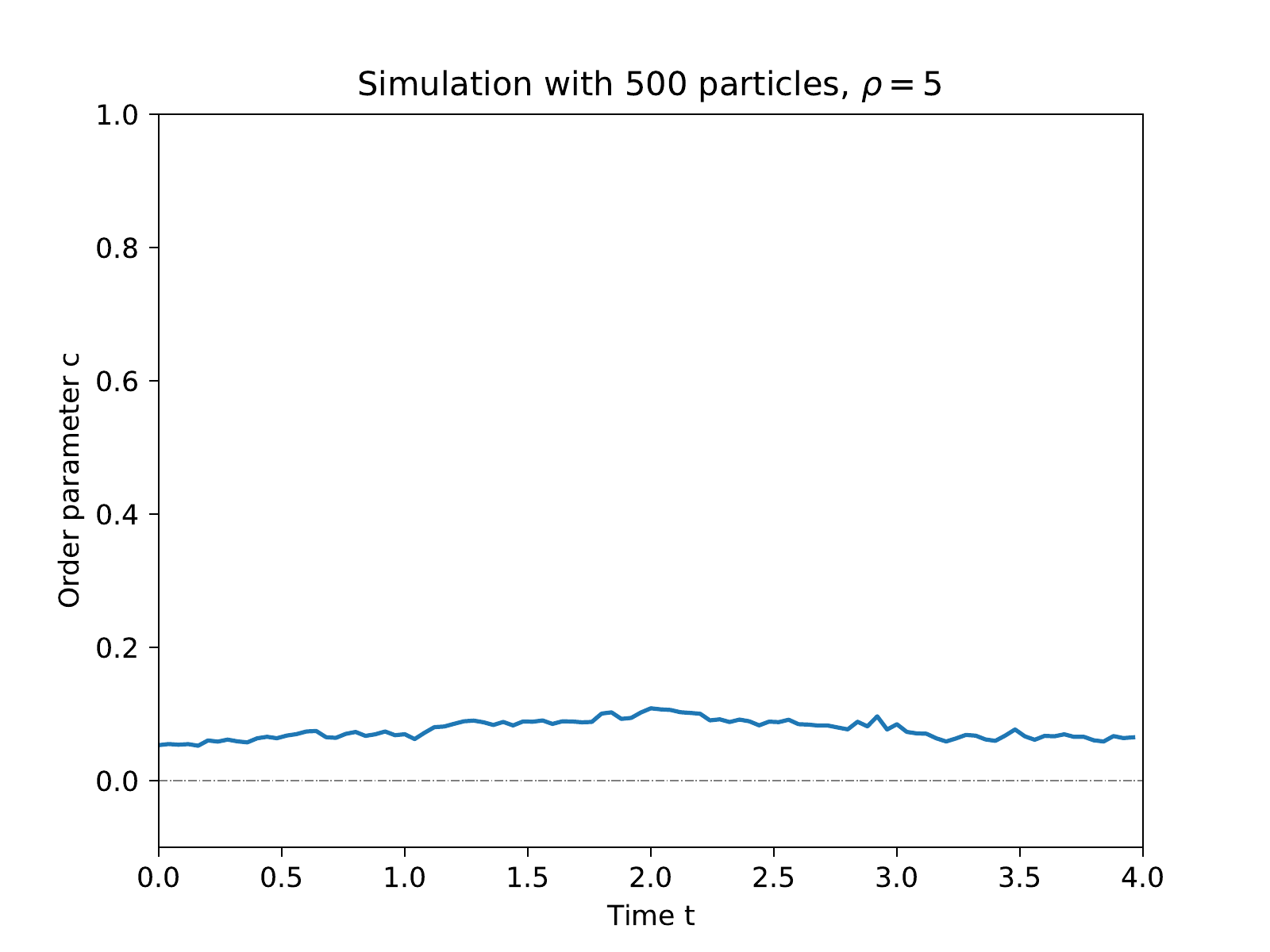}
\caption{\label{fig-two-examples}Time evolution of the order parameter in four situations.}
\end{center}
  
\end{figure}

In order to obtain a more precise illustration of this phenomenon, we ran~$500$ such simulations with various values of the parameter~$ρ$ and different initial conditions\footnote{For a better illustration, the parameter~$ρ$ and the initial order parameter~$c$ are not uniformly sampled, in order to see more points in the region of interest.}, still with~$N=500$ and~$Δt=0.04$, for~$500$ time iterations. Figure~\ref{fig-scatterplots} depicts the initial order parameters~$c$ and strengths~$ρ$, and their value after~$500$ iterations ($t=20$). We clearly see two thresholds for~$ρ$. The first threshold that we will denote~$ρ^*$, is such that for all simulations with~$ρ<ρ^*$, the order parameter seems to be close to~$0$ for large times. The second threshold, that we will denote~$ρ_c$ (with~$ρ^*<ρ_c$), is such that for all simulations with~$ρ>ρ_c$, the order parameter does not stay close to~$0$ for large times, and stabilizes around a quite high value. In the intermediate regime~$ρ^*<ρ<ρ_c$, both behaviours occur. This is what is called first-order (or discontinuous) phase transition: the order parameter does not vary continuously when going from one behaviour to the other.  

\begin{figure}[h!]
  \begin{center}
\includegraphics[width=7cm]{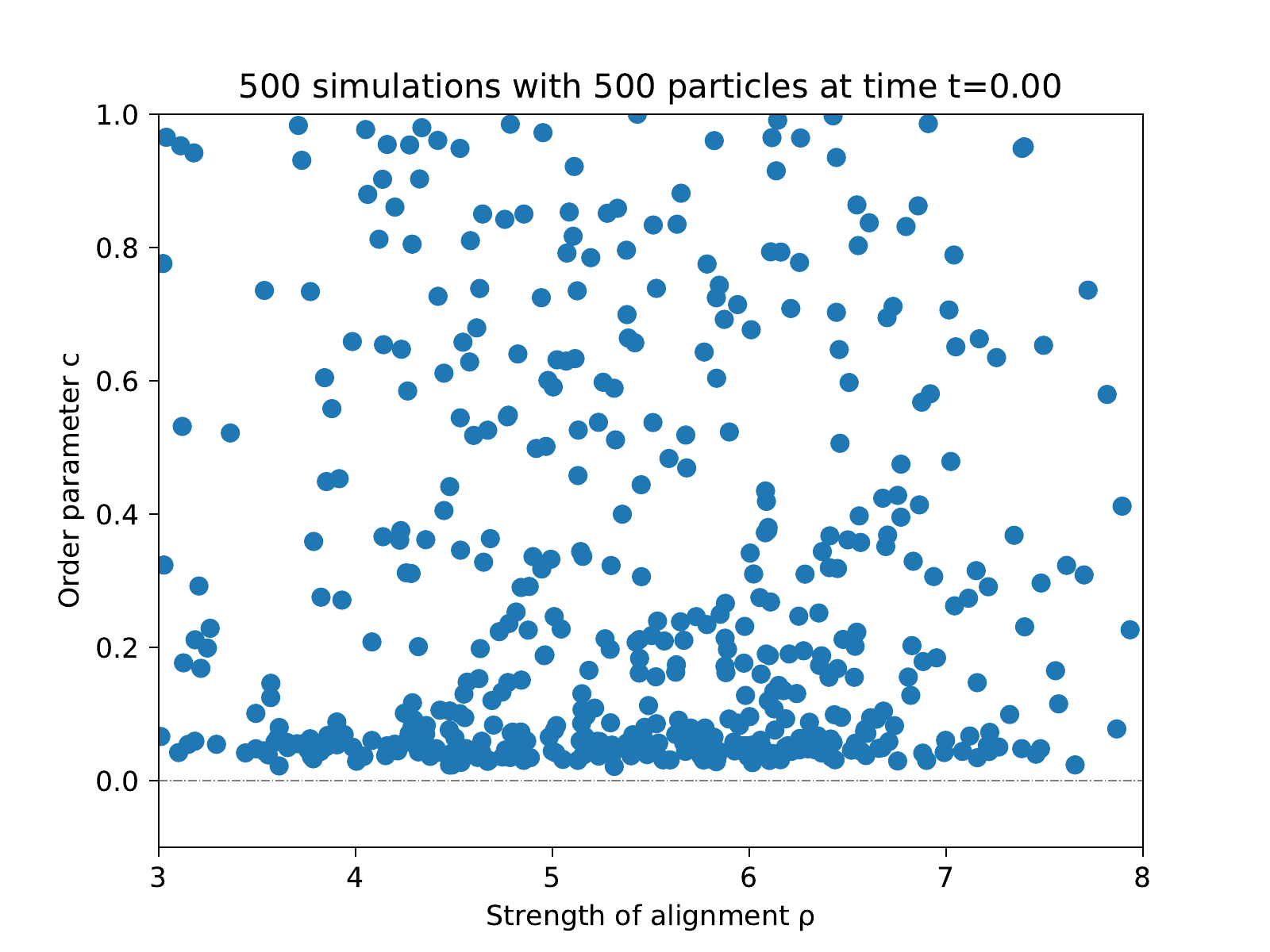}
\includegraphics[width=7cm]{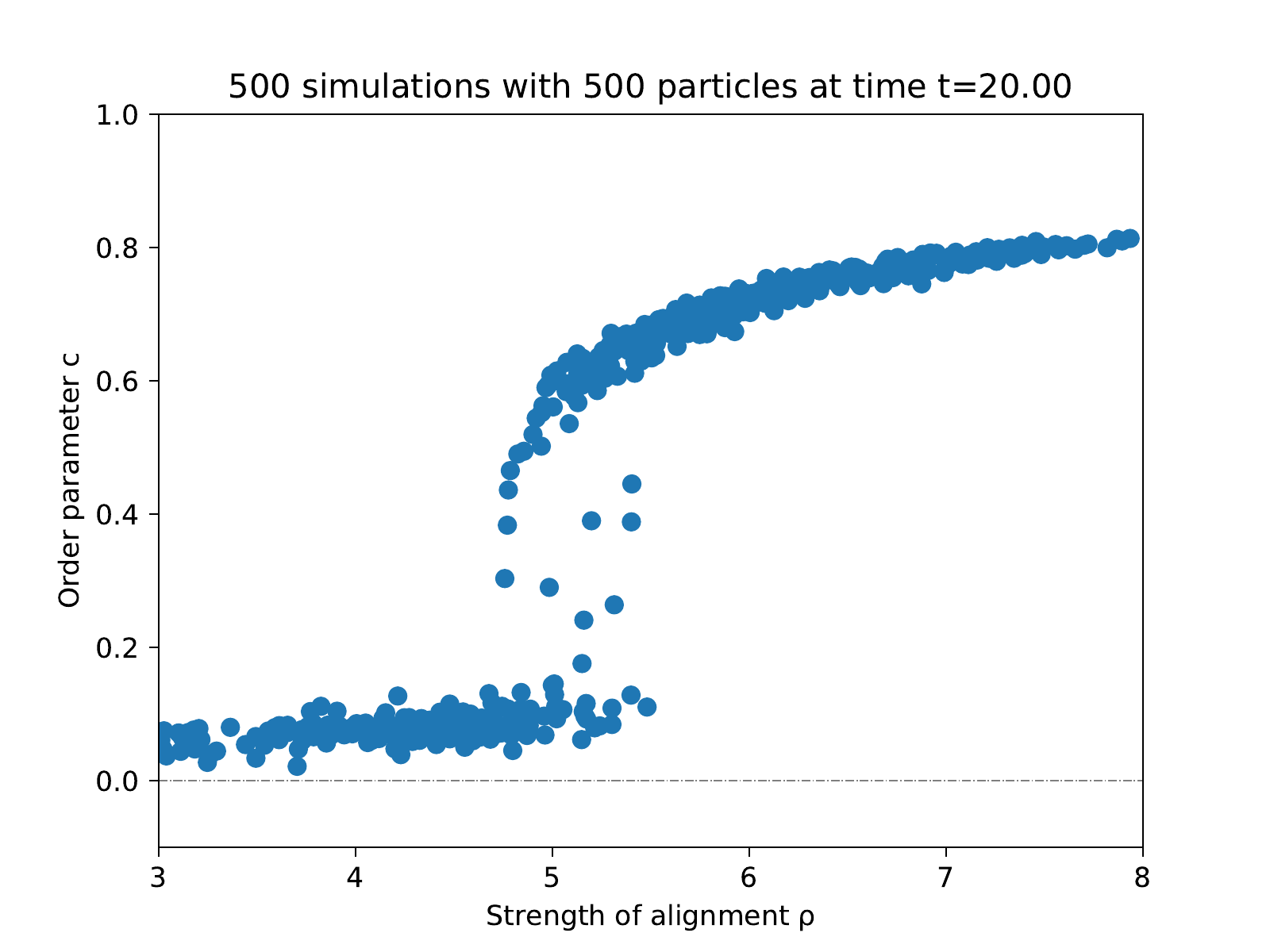}
\caption{\label{fig-scatterplots} Numerical illustration of a first-order phase transition.}
\end{center}
\end{figure}

The aim of the next sections is to present a rigorous mathematical description of this phenomenon in the framework of a kinetic equation corresponding to the limiting behaviour of the system of SDEs~\eqref{SDEsystem} when~$N→∞$. 

\section{Mean-field limit and compatibility equation}
\label{sec-meanfield}
Let us first consider the first part of the system~\eqref{SDEsystem}, as if~$t∈ℝ↦J(t)∈M_3(ℝ)$ was a prescribed regular function:
\begin{equation}
  \label{SDEJ}
  \d A=P_{T_{A}}J \d t + 2 P_{T_{A}}\circ \d B_{t}.
\end{equation}
As before for the simple SDE~\eqref{simpleSDE}, the law~$t↦μ(t,·)$ of such a stochastic process would satisfy the following (linear) Fokker--Planck equation:
\begin{equation}
  \label{FPJ}
  ∂_tμ=-∇_A·(μ\,P_{T_A}J)+Δ_Aμ=∇_A·\Big[M_J(A)∇_A\Big(\frac{μ}{M_J(A)}\Big)\Big],
\end{equation}
where the definition of the generalized von Mises distribution~$M_J$ is given by the formula~\eqref{defVM}. Let us now suppose that several such processes~$A_k$ satisfying the SDE~\eqref{SDEJ} were independently drawn, with different independent Brownian motions~$B_{t,k}$, and independent initial conditions following a probability measure~$μ_0$ on~$SO_3$. Their law at time~$t$ would be given by~$μ(t,·)$, solution of the Fokker--Planck equation~\eqref{FPJ} with initial condition~$μ_0$ by the law of large numbers the average~$\frac1{N}\sum_{k=1}^NA_k(t)$ would converge to the expectation of one of this process, that we call~$\mathcal{J}[μ(t,·)]$. More generally, we define~$\mathcal{J}[f]$ for any finite measure~$f$ on~$SO_3(ℝ)$ (not necessarily a probability measure, it may also be a signed measure):
\begin{equation}
  \label{def-J}
  \mathcal{J}[f]=∫_{SO_3(ℝ)}A\,f(A)\d A.
\end{equation}

To deal with the system~\eqref{SDEsystem}, where~$J(t)=\frac{ρ}{N}\sum_{k=1}^NA_k(t)$ is not prescribed but depends on all the particles, we cannot expect the particles~$A_k$ to behave independently. However one can show that in the limit~$N→∞$, their behaviour is close to independent particles. This is called the propagation of chaos property, and we refer to~\cite{sznitman1991topics} for an introduction on this subject. One of the typical results in this theory is that the empirical measure of the particle system converges to a solution to the (now nonlinear) Fokker--Planck equation corresponding to~\eqref{FPJ} with~$J(t)=ρ\mathcal{J}[μ(t,·)]$:

\begin{proposition} If~$A_{k,0}$ are independent random rotation matrices distributed according to the probability measure~$μ_0$, then the empirical measure~$μ^N(t)=\frac1{N}\sum_{k=1}^Nδ_{A_k(t)}$ associated to the solution of the system of SDEs~\eqref{SDEsystem} converges (in Wasserstein distance) to the solution~$μ$ of the following nonlinear Fokker--Planck equation, with initial condition~$μ_0$ :
\begin{equation}
  \label{FPJmu}
  ∂_tμ=-ρ\,∇_A·(μ\,P_{T_A}\mathcal{J}[μ])+Δ_Aμ.
\end{equation}
The convergence is uniform on~$[0,T]$ for all~$T>0$.
\end{proposition}
\begin{proof}
  We will not provide the proof in detail here, as it follows the classical theory of propagation of chaos for coupled drift-diffusion processes, but we will recall some important steps. It has to be adapted to the framework of SDEs on a manifold, but this is not a real problem in this compact case (see for instance~\cite{bolley2012meanfield} in the case of the Vicsek model on the sphere). Let us recall the coupling argument such as the one in~\cite{sznitman1991topics}. We start by proving the well-posedness of this following SDE (the coupling process):
  \begin{equation}
    \label{couplingProcess}
    \begin{cases}\d A = ρ\, P_{T_{A}}\mathcal{J}[μ] \d t + 2 P_{T_{A}}\circ \d B_{t},\\μ(t,·)\text{ is the law of } A(t).
    \end{cases}
  \end{equation}
  The proof of this well-posedness, seen as a fixed point problem (either for the function~$J(t)=ρ\mathcal{J}[μ]$ or directly on the law~$μ$) is done thanks to a Picard iteration which leads to a contraction in the appropriate Wasserstein metric.

  We then construct independent solutions to this coupling process~$\overline{A}_k$ with independent Brownian motions~$B_{t,k}$ and initial conditions~$A_{k,0}$: the same as the Brownian motions and initial conditions used for the original system of SDEs~\eqref{SDEsystem}. All these processes~$\overline{A}_k$ have the same law, which is the solution~$μ$ of the Fokker--Planck equation~\eqref{FPJmu} starting with~$μ_0$. By the law of large numbers, the empirical distribution~$\overline{μ}^N$ of the coupling processes converges to~$μ$, and therefore it is enough to estimate the distance between~$\overline{μ}^N$ and~$μ^N$. This can be done by obtaining estimates of the form
  \begin{equation}
    \label{couplingEstimates}
    𝔼[∥A_k-\overline{A}_k∥^2]⩽\frac{\exp(CT)}{N},
  \end{equation}
  for all~$1⩽k⩽N$, which gives control on the~$2$-Wasserstein distance between~$\overline{μ}^N$ and~$μ^N$ on the time interval~$[0,T]$.
\end{proof}

We know want to study the long time behaviour of the nonlinear Fokker--Planck equation~\eqref{FPJmu}, that we will rewrite in function of~$f=ρ\,μ$ (in that case,~$ρ$ represents the total “mass” of~$f$). Since~$ρ\mathcal{J}[μ]=\mathcal{J}[f]$, it therefore has the following form, without any parameter on the equation:
\begin{equation}
  \label{FPJf}
  ∂_tf=-∇_A·(f\,P_{T_A}\mathcal{J}[f])+Δ_Af.
\end{equation}
This is an equation of the form~$∂_tf=\mathcal{C}[f]$ where~$\mathcal{C}[f]$ can also be written, using the definition~\eqref{defVM} of the von Mises distribution~$M_J$, under the following factorized form:
\begin{equation*}
  \label{def-Cf}
  \mathcal{C}[f]=∇_A·\Big[M_{\mathcal{J}[f]}(A)∇_A\Big(\frac{f}{M_{\mathcal{J}[f]}(A)}\Big)\Big].
\end{equation*}
In order to understand the long time behaviour of the solution, let us first look at stationary solutions.
\begin{proposition}\label{prop-equilibria}
  A measure~$f$ on~$SO_3(ℝ)$ is a stationary solution of the Fokker--Planck equation~\eqref{FPJf} if and only if it is of the form~$f=ρ M_J$, where~$J$ satisfies the following compatibility equation
\begin{equation}
  \label{compatJ}
  J=ρ\mathcal{J}[M_{J}].
\end{equation}
\end{proposition}
\begin{proof}Since we have, by integration by parts,
\begin{equation*}
  ∫_{SO_3(ℝ)}\frac{f}{M_{\mathcal{J}[f]}(A)}\,\mathcal{C}[f]\d A=-∫_{SO_3(ℝ)}\Big{∥}∇_A\Big(\frac{f}{M_{\mathcal{J}[f]}(A)}\Big)\Big{∥}^2M_{\mathcal{J}[f]}(A)\d A,
\end{equation*}
we immediately get that if~$\mathcal{C}[f]=0$ then~$f$ has to be proportional to~$M_{\mathcal{J}[f]}$, and the total mass of~$f$, denoted by~$ρ$, gives the coefficient of proportionality. Then, taking the average on~$SO_3(ℝ)$ against~$A$, thanks to the definition~\eqref{def-J} of~$\mathcal{J}$, we obtain, denoting~$J=\mathcal{J}[f]$:
\begin{equation*}
  J=\mathcal{J}[f]=\mathcal{J}[ρM_{\mathcal{J}[f]}]=ρ\mathcal{J}[M_{J}],
\end{equation*}
which is the compatibility equation for~$J$. Conversely, if~$J$ is a fixed point of this map~$J↦ρ\mathcal{J}[M_{J}]$, then setting~$f=ρ M_J$, we get~$\mathcal{J}[f]=J$, and then~$\mathcal{C}[f]=0$.
\end{proof}

Before obtaining a simple characterization of the solutions of the compatibility equation~\eqref{compatJ}, which is the object of the next section, let us give some more results on the solutions to the Fokker--Planck equation~\eqref{FPJf}.

\begin{proposition} For all nonnegative measure~$f_0$ on~$SO_3(ℝ)$, with total mass~$ρ>0$, there exists a unique weak solution~$f$ to the nonlinear Fokker--Planck equation~\eqref{FPJf} such that~$f(t,·)$ converges to~$f_0$ (in Wasserstein distance) as~$t→0$. This solution belongs to~$C^∞((0,+∞),SO_3(ℝ))$ and is positive for any positive time. Furthermore, we have the following uniform estimates in time: for all~$t_0>0$, and~$s∈ℝ$, the solution~$f$ is uniformly bounded on~$[t_0,+∞)$ in the the Sobolev space~$H^s(SO_3(ℝ))$. 
\end{proposition}
The proof of this proposition can be obtained through simple energy estimates in~$H^s(SO_3(ℝ))$, using Poincaré inequalities for high modes and the fact that the low modes are uniformly bounded in time. Indeed, the nonlinearity in the Fokker--Planck equation~\eqref{FPJf} is only through~$\mathcal{J}[f]$, which is uniformly bounded thanks to its definition~\eqref{def-J} and the fact that~$SO_3(ℝ)$ is compact, together with the fact that the total mass~$ρ$ is preserved. The positivity comes from the maximum principle. We refer to~\cite{frouvelle2012dynamics} to a detailed proof of such results on the unit sphere instead of~$SO_3(ℝ)$, for which all the arguments may be used similarly.

Let us now describe the free energy associated to this Fokker--Planck equation, which may be rewritten
\[∂_tf=∇_A·\big(f\,∇_A(\ln f - A·\mathcal{J}[f])\big).\]
Multiplying by~$\ln f - A·\mathcal{J}[f]$ and integrating over~$SO_3(ℝ)$, the left-hand side of the equality can be seen as a time derivative, and the right-hand side can be integrated by parts, to obtain the following dissipation relation:
\begin{equation}\frac{\d}{\d t}\mathcal{F}[f]+\mathcal{D}[f]=0\label{dissipation},
\end{equation}
where
\begin{align}
  \label{def-F}
  \mathcal{F}[f]=∫_{SO_3(ℝ)}f(A)\ln f(A)\d A -\frac12∥\mathcal{J}[f]∥^2,\\
  \label{def-D}
  \mathcal{D}[f]=∫_{SO_3(ℝ)}f(A)∥∇_A(\ln f - A·\mathcal{J}[f])∥^2\d A.
\end{align}
We can then prove, as in~\cite{frouvelle2012dynamics} that being a stationary state of the Fokker--Planck equation (see Proposition~\ref{prop-equilibria}) is equivalent to be a critical point of~$\mathcal{F}$ under the constraint of mass~$ρ$, and that is also equivalent to be a function with no dissipation~$(\mathcal{D}[f]=0)$.

We then have a decreasing free energy~$\mathcal{F}[f]$, and thanks to a kind of LaSalle’s principle, we obtain that the solution converges to a set of equilibria: 
\begin{proposition}\label{prop-lasalle} Let~$f_0$ be a nonnegative measure on~$SO_3(ℝ)$ with mass~$ρ>0$. We denote by~$\mathcal{F}_∞$ the limit of~$\mathcal{F}[f(t,·)]$ as~$t→+∞$, where~$f$ is the solution to the Fokker--Planck equation~\eqref{FPJf} with initial condition~$f_0$. Then the set of equilibria~$\mathcal{E}_∞$, given by
  \begin{equation*}\mathcal{E}_∞=\{ρ M_{J}\text{ such that }J=ρ\mathcal{J}[M_J]\text{ and }\mathcal{F}[ρM_{J}]=\mathcal{F}_∞\},
  \end{equation*}
  is not empty. Furthermore, the solution~$f$ converges in any Sobolev space~$H^s$ to this set of equilibria in the following sense:
  \[\lim_{t→∞}\inf_{g∈\mathcal{E}_∞}∥f(t,·)-g∥_{H^s}=0.\]
\end{proposition}
Once more, the proof of this proposition follows exactly the one given in~\cite{frouvelle2012dynamics}. The important point of this proposition is that once the structure of the solutions of the compatibility equation~\eqref{compatJ} is known (which is the aim of the next section), it gives a lot of information on the large time behaviour of the solutions to the Fokker--Planck equation.

Before giving a precise description of these solutions, let us remark that~$J=0$ is always a solution to the compatibility equation, since~$\mathcal{J}[ρ]=0$, therefore the uniform distribution with mass~$ρ$ is a steady-state. We want to expand the free energy~$\mathcal{F}$ around this steady-state. 
We will need the following lemma (Lemma~$3.3$ of~\cite{degond2020phase}): 
\begin{lemma}For all~$J∈M_3(ℝ)$,
  \begin{equation}
    \label{JAA}∫_{SO_3(ℝ)}(J·A) A\, \d A=\frac16J.
  \end{equation}
\end{lemma}
Consequently, if~$f$ is a finite measure, the orthogonal projection of~$f$ on the space of functions of the form~$A↦J·A$ for~$J∈M_3(ℝ)$ is given by~$A↦6\mathcal{J}[f]·A$. Now, let us take a nonnegative measure~$f$ with mass~$ρ$, we write~$J=\mathcal{J}[f]$ and~$g(A)=6\,J·A$. We suppose that~$∥J∥$ is sufficiently small, so that~$ρ+g>0$ on~$SO_3(ℝ)$. We write~$h=f-ρ-g$, so~$h$ is a finite measure with zero average and~$\mathcal{J}[h]=0$. Then we obtain, by convexity of~$x↦x\ln x$ on~$ℝ_+$:
\begin{align}
  \mathcal{F}[f]&⩾∫_{SO_3(ℝ)}[(ρ+g(A))\ln (ρ+g(A))+h(A)(\ln(ρ+g(A))+1)]\,\d A-\tfrac12∥J∥^2\nonumber\\
                &⩾\mathcal{F}[ρ+g]+∫_{SO_3(ℝ)}h(A)\Big(1+\ln ρ+\frac{g(A)}{ρ}\Big)\d A-O(∥g∥_∞^2)∫_{SO_3(ℝ)}|h(A)|\d A.\nonumber\\
  &⩾\mathcal{F}[ρ+g]-O(∥J∥^2)\Big(∫_{SO_3(ℝ)}|f(A)-ρ|\,\d A + O(∥J∥)\Big).\label{Ff2}
\end{align}
Next we compute
\begin{align}
  \mathcal{F}[ρ+g] &=ρ\ln ρ+\frac1{2ρ}∫_{SO_3(ℝ)}(6A·J)^2\d A + O(∥g∥_∞^3) - \tfrac12∥J∥^2\nonumber\\
  &=\mathcal{F}[ρ] + \frac{6-ρ}{2ρ}∥J∥^2 + O(∥J∥^3),\label{Frhog}
\end{align}
thanks to Lemma~\ref{JAA}. We therefore see that the sign of~$6-ρ$ plays a role to study the nature, as a critical point of~$\mathcal{F}$, of the uniform distribution of mass~$ρ$:
\begin{proposition} We set~$ρ_c=6$. \label{prop-stability-uniform}
  \begin{itemize}
  \item If~$ρ<ρ_c$, then the uniform distribution with mass~$ρ$ is a local strict minimizer of the free energy~$\mathcal{F}$ under the constraint of total mass~$ρ$.
    \item If~$ρ>ρ_c$, the uniform distribution with mass~$ρ$ is not a local minimizer of the free energy~$\mathcal{F}$ under the constraint of total mass~$ρ$. 
    \end{itemize}
  \end{proposition}
\begin{proof}
  When~$ρ>6$, it is clear thanks to~\eqref{Ff2} and~\eqref{Frhog} that if~$∥J∥$ and~$∫_{SO_3(ℝ)}|f-ρ|$ are sufficiently small and~$J≠0$, then~$\mathcal{F}[f]>\mathcal{F}[ρ]$. If~$J=0$ but~$f≠ρ$, then by strict convexity of~$x↦x\ln x$ on~$ℝ_+$, we get
  \[\mathcal{F}[f]=∫_{SO_3(ℝ)}f(A)\ln f(A)\,\d A>∫_{SO_3(ℝ)}(ρ\ln ρ+(f(A)-ρ)\ln ρ)\,\d A=\mathcal{F}[ρ].\]
  The second point follows directly from~\eqref{Frhog}.
\end{proof}

This last proposition gives an insight on the stability of the uniform steady-state (we will indeed see later that this uniform steady-state is isolated). In summary, we have shown that there is a phenomenon of phase transition at the threshold~$ρ=ρ_c$, and we know thanks to Proposition~\ref{prop-lasalle} that there must exist other types of steady-states, at least when~$ρ>ρ_c$. We are now ready to give a precise description of those non-isotropic equilibria.

\section{Link with higher dimensional polymers, solutions to the compatibility equation}
\label{sec-polymers}
This section is the summary of the results we obtained in~\cite{degond2020phase} to solve the compatibility equation~\eqref{compatJ} (in a slightly different context, see Section~\ref{sec-BGK}), therefore we will omit the proofs.

Let us first recall some definitions. We denote by~$ℍ$ the set of quaternions: objects of the form~$q=a+b\mathbf{i}+c\mathbf{j}+d\mathbf{k}$, where~$(a,b,c,d)∈ℝ^4$ and the imaginary quaternions satisfy~$\mathbf{i}^2=\mathbf{j}^2=\mathbf{k}^2=\mathbf{i}\mathbf{j}\mathbf{k}=-1$. For such a quaternion~$q$, we denote by~$q^*=a-b\mathbf{i}-c\mathbf{j}-d\mathbf{k}$ its conjugate. It satisfies~$qq^*=q^*q=a^2+b^2+c^2+d^2=|q|^2$, if we identify the Euclidean space~$ℝ^4$ with~$ℍ$. We denote then by~$ℍ_1$ the set of units quaternions: those for which~$|q|^2=1$.

We say that a quaternion~$q$ of the previous form is purely imaginary if its real part~$a$ is zero. It allows now to identify~$ℝ^3$ with the set of purely imaginary quaternions. We will use boldface letters when using this identification.

The first proposition is a link between~$SO_3(ℝ)$ and~$ℍ_1/\{±1\}$.
\begin{proposition}\label{prop-quaternions-rotations}For any~$q∈ℍ_1$, the linear map~$\mathbf{u}↦q\mathbf{u}q^*$ sends purely imaginary quaternions on purely imaginary quaternions of the same norm. It is therefore identified as a rotation of~$ℝ^3$, and the corresponding rotation matrix is denoted~$Φ(q)$. Conversely for any rotation matrix~$A∈SO_3(ℝ)$, there exists a unit quaternion~$q$ such that~$A=Φ(q)$ (this quaternion is not unique, the only other possibility being~$-q$). The map~$Φ$ can then be seen as a group isomorphism between~$SO_3(ℝ)$ and~$ℍ_1$ (this is actually a local isometry between the manifolds). In practice, the matrix~$R(θ,\mathbf{n})$ given by Rodrigues’ formula~\eqref{eqRodrigues} corresponds to the quaternion~$q=\cos(\frac{θ}2)+\sin(\frac{θ}2)\mathbf{n}$ (remember that vectors in~$ℝ^3$ are seen as purely imaginary quaternions, and remark that if we replace~$θ$ by~$θ+2π$, we get the same rotation matrix, but the opposite quaternion). 
\end{proposition}

This allows to represent a rotation matrix by a unit quaternion up to multiplication by~$±1$. This is reminiscent of describing rodlike polymers as unit vectors up to multiplication by~$±1$, but generalized in dimension~$4$. This analogy was the starting point of our work~\cite{degond2018quaternions}, where we used those unit quaternions for the modeling of alignment of rigid bodies. In the present case, we will see that this analogy will actually be very helpful, by transforming the compatibility equation~\eqref{compatJ} into another one which has already been solved in~\cite{wang2008unified}, in the context of suspensions of diluted polymers.

We denote by~$\mathcal{S}_4^0(ℝ)$ the space of symmetric and trace-free matrices of dimension~$4$, which are called~$Q$-tensors. To a unit quaternion~$q$, we can associate the~$Q$-tensor given by~$q⊗q-\frac14I_4$. Remark that two unit quaternions~$q$ and~$\widetilde q$ are associated to the same~$Q$-tensor if and only if~$q=±\widetilde q$ (this is a unit vector in the eigenspace of this~$Q$-tensor associated to the eigenvalue~$\frac34$, which is one-dimensional). So we have another way to represent unit quaternions up to multiplication by~$±1$ in this space. The important fact to notice is that those two embeddings are actually the same, up to a linear isomorphism between the spaces~$M_3(ℝ)$ and~$\mathcal{S}_4^0(ℝ)$, which has nice properties.

\begin{proposition}\label{isomorphism-Qtensors} There exists a linear isomorphism~$ϕ$ between the spaces~$M_3(ℝ)$ and~$\mathcal{S}_4^0(ℝ)$ (both of dimension~$9$) with the following properties:
  \begin{gather}
    ∀q∈ℍ_1,\quad ϕ(Φ(q))=q⊗q-\tfrac14I_4,\label{PhiQ}\\
    ∀J∈M_3(ℝ),∀q∈ℍ_1, \quad \frac12J·Φ(q)=q·ϕ(J)q,\label{defphi}
  \end{gather}
  where the map~$Φ$ is given by Proposition~\ref{prop-quaternions-rotations}. The dot product in the left-hand side of~\eqref{defphi} is the metric in the space~$M_3(ℝ)$ given in~\eqref{dotSO3}, while the one in the right-hand side is the canonical scalar product of~$ℝ^4$. Furthermore, the isomorphism~$ϕ$ preserves the diagonal structure:~$J∈M_3(ℝ)$ is diagonal if and only if~$ϕ(J)$ is diagonal in~$\mathcal{S}_4^0(ℝ)$.
\end{proposition}

The proof of this proposition is done in~\cite{degond2020phase}. The expression~\eqref{defphi} is actually the definition of~$ϕ$: the left-hand side is a quadratic form in~$q$ (seen as an element of~$ℝ^4$), defined for any unit quaternion, which defines a symmetric bilinear form on all quaternions, the matrix of which is~$ϕ(J)$. The expression of~$ϕ(J)$ is given in the appendix of~\cite{degond2020phase}, which gives the fact that it is bijective and with values in trace-free matrices, and the provides the property~\eqref{PhiQ}. With this isomorphism, we can rewrite the compatibility equation in the framework of~$Q$-tensors. For a finite measure~$f$ on~$ℍ_1$, we define its averaged~$Q$-tensor by
\begin{equation*}\label{defQf}
  \mathcal{Q}[f]=∫_{ℍ_1}f(q) (q⊗q-\tfrac14I_4)\d q.
\end{equation*}
Therefore, thanks to the definition~\eqref{def-J} of~$\mathcal{J}$ and the fact that~$Φ$ is a local isometry, we obtain, for a finite measure~$f$ on~$SO_3(ℝ)$
\[ϕ(\mathcal{J}[f])=∫_{SO_3(ℝ)}ϕ(A)f(A)\d A=∫_{ℍ_1}ϕ(Φ(q))f(Φ(q))\d q=\mathcal{Q}[f∘Φ].\]
Finally, we also define the generalized von Mises associated to~$Q∈\mathcal{S}_4^0(ℝ)$ by
\begin{equation*}\label{defVMQ}
  M_Q(q)=\frac1{\mathcal{Z}(Q)}\exp(q·Qq)\text{, where }\mathcal{Z}(Q)=∫_{ℍ_1}\exp(q·Qq)\d q,
\end{equation*}
where we use the same notation as in~\eqref{defVM} for the generalized von Mises on~$SO_3(ℝ)$, but it will always be clear following the context which definition is concerned. Using the property~\eqref{defphi}, it is then clear that~$M_J(Φ(q))=M_{2ϕ(J)}(q)$. Therefore, the compatibility equation~\eqref{compatJ} becomes, writing~$Q=2ϕ(J)$:

\[Q=2ϕ(J)=2ρ\,ϕ(\mathcal{J}[M_J])=2ρ\,\mathcal{Q}[M_Q].\]

It happens that this equation is exactly the compatibility equation that we obtain when we try to obtain the steady states of the following Fokker--Planck equation, for a probability measure~$μ$ on~$ℍ_1$:
\[∂_tμ=-2ρ\, ∇_q·(μ ∇_q(q·\mathcal{Q}[μ]q))-Δ_qμ.\]
This corresponds to the Smoluchowski (or Doi--Onsager) equation for suspensions of dilute rodlike polymers with Maier--Saupe potential of strength~$2ρ$, and is nothing else than our Fokker--Planck equation~\eqref{FPJ}, up to a change of variable thanks to the map~$Φ$. It happens that this compatibility equation has been studied a lot in dimension~$3$ (instead of~$4$ here), with the independent works~\cite{constantin2004asymptotic,fatkullin2005critical,liu2005axial}. And in the work~\cite{wang2008unified}, a unified approach has been proposed, which allows to treat the case of higher dimensional space. The main result is that a solution~$Q∈\mathcal{S}_n^0(ℝ)$ of the compatibility equation~$Q=α\,\mathcal{Q}[M_Q]$ can have at most two different eigenvalues. In dimension~$4$, it means that if~$Q$ is different from zero, there are only two cases: either one eigenvalue is simple and the other one is triple, or both are double. In the first case, if we take~$q$ a unit quaternion in the eigenspace of dimension one, we get that~$Q$ is proportional to~$q⊗q-\frac14I_4$, which means that~$J=ϕ^{-1}(Q)$ is proportional to the rotation matrix~$Φ(q)$. And indeed it is possible to see that if~$A_0$ is a rotation matrix and~$α∈ℝ$, then~$\mathcal{J}[M_{αA_0}]$ is proportional to~$A_0$, with a coefficient~$c_1(α)$ (that can be expressed using an appropriate volume form on~$SO_3(ℝ)$ and will be given later on). Therefore the compatibility equation~\eqref{compatJ} becomes the one-dimensional equation~$α=ρc_1(α)$. For the second case, it is a little bit more subtle, but it still leads to a one-dimensional equation of the form~$α=ρc_2(α)$. The results are summarized in the following proposition (corresponding to Theorem~$5$ of~\cite{degond2020phase}):

\begin{proposition}\label{prop-solutions-compat}
  The solutions to the compatibility equation~\eqref{compatJ} are:
  \begin{itemize}
  \item The matrix~$J=0$,
  \item the matrices of the form~$J=αA_0$ with~$A_0∈SO_3(ℝ)$ and where~$α∈ℝ\setminus\{0\}$ satisfies the scalar compatibility equation
    \begin{equation}α=ρc_1(α),\label{compat1}
    \end{equation}
  \item the matrices of the form~$J=α\sqrt3\,\mathbf{a}_0⊗\mathbf{b}_0$ where~$\mathbf{a}_0$ and~$\mathbf{b}_0$ are two unit vectors of~$ℝ^3$ and~$α>0$ satisfies the scalar compatibility equation
    \begin{equation}α=ρc_2(α),\label{compat2}
    \end{equation}
       
  \end{itemize}
  with the functions~$c_1$ and~$c_2$ given by
  \begin{align*}
    c_1(α)&=\frac{∫_0^π\frac13(2\cosθ+1)\sin^2(\frac{θ}2)\exp(α\cosθ)\d θ}{∫_0^π\sin^2(\frac{θ}2)\exp(α\cosθ)\d θ},\\
    c_2(α)&=\frac1{\sqrt3}\frac{∫_0^π\cosφ\sinφ\exp(\frac{\sqrt3}2α\cosφ)\d φ}{∫_0^π\sinφ\exp(\frac{\sqrt3}2α\cosφ)\d φ}.
  \end{align*}

\end{proposition}
Compared to the convention taken in~\cite{degond2020phase}, we chose to add the constant~$\sqrt3$ in the last type of solutions (changing accordingly the expression of~$c_2(α)$). The reason is that if~$J$ is a solution to the compatibility equation~\eqref{compatJ}, where~$α$ satisfies~\eqref{compat1} or \eqref{compat2}, then~$∥J∥^2=\frac32α^2$. The order parameter~$c$ associated to the steady state~$ρM_J$ by the formula~\eqref{def-c} is then equal to~$\frac{|α|}{ρ}$ which is~$|c_1(α)|$ or~$|c_2(α)|$. These functions~$c_1$ and~$c_2$ then provide the values of the order parameter of the considered steady-state. The study of these functions (and more precisely the behaviour of~$\frac{α}{c_1(α)}$ and~$\frac{α}{c_2(α)}$) is the key to provide a complete description of the possible steady-states. Once more, the following proposition is taken from~\cite{degond2020phase}.

\begin{proposition} The functions~$c_1$ and~$c_2$ are both strictly increasing on~$ℝ$ having value~$0$ at~$0$. Therefore~$0$ is always a solution to the scalar compatibility equations~\eqref{compat1} and~\eqref{compat2}. If we set~$ρ_c=6$, then when~$α→0$, the functions~$ρ_1:α↦\frac{α}{c_1(α)}$ and~$ρ_2:α↦\frac{α}{c_2(α)}$ both have a limit equal to~$ρ_c$. Furthermore:
  \begin{itemize}
  \item There exists~$α^*>0$ such that~$ρ_1$ is decreasing on~$(-∞,α^*]$ and increasing on~$[α^*,+∞)$, converging to~$+∞$ at~$±∞$. We set~$ρ^*=ρ_1(α^*)$ (which is less than~$ρ_c$). For all~$ρ⩾ρ^*$, we define~$α_1^{↑}(ρ)$ (resp.~$α_1^{↓}(ρ)$) to be the unique value of~$α⩾α^*$ (resp~$α⩽α^*$) such that~$ρ_1(α)=ρ$. Finally, we define~$\widetilde{c}_1^{↑}(ρ)=c_1(α_1^{↑}(ρ))$ and~$\widetilde{c}_1^{↓}(ρ)=c_1(α_1^{↓}(ρ))$. Setting~$c^*=c_1(α^*)$, the function~$\widetilde{c}_1^{\uparrow}$ (resp.~$\widetilde{c}_1^{\downarrow}$) is increasing (resp. decreasing) on~$[ρ^*,+∞)$, with value~$c^*$ at~$ρ^*$, and converging to~$1$ (resp.~$-\frac{1}3$) at~$+∞$.

    Numerically, we obtain~$α^*≈1.9395$,~$ρ^*≈4.5832$, and~$c^*≈0.4232$.
    \item The function~$ρ_2$ is (even and) increasing on~$[0,+∞)$, converging to~$+∞$ at~$+∞$. For all~$ρ⩾ρ_c$, we define~$α_2(ρ)$ to be the unique value of~$α⩾α^*$ such that~$ρ_2(α)=ρ$. Finally, we define~$\widetilde{c}_2(ρ)=c_2(α_2(ρ))$. The function~$\widetilde{c}_2$ is increasing on~$[ρ_c,+∞)$, with value~$0$ at~$ρ_c$ and converging to~$\frac1{\sqrt3}$ at~$+∞$.  
    \end{itemize}
  \end{proposition}

Figure~\eqref{fig-plots-c} depicts a plot of these functions~$\widetilde{c}_1^{↑}$ (solid),~$\widetilde{c}_1^{↓}$ (dashed), and~$\widetilde{c}_2$ (dashed-dot line), in log-scale for~$ρ∈[2,40]$. They represent the order parameters (up to sign) of the different families of steady-states. We also drew a solid line at level~$0$ for~$ρ<ρ_c$ and a dotted line at level~$0$ for~$ρ>ρ_c$, corresponding to the order parameter of the uniform steady-state (and illustrating the result of Proposition~\ref{prop-stability-uniform} regarding its stability). 
\begin{figure}[h!]
  \begin{center}
\includegraphics[width=14cm]{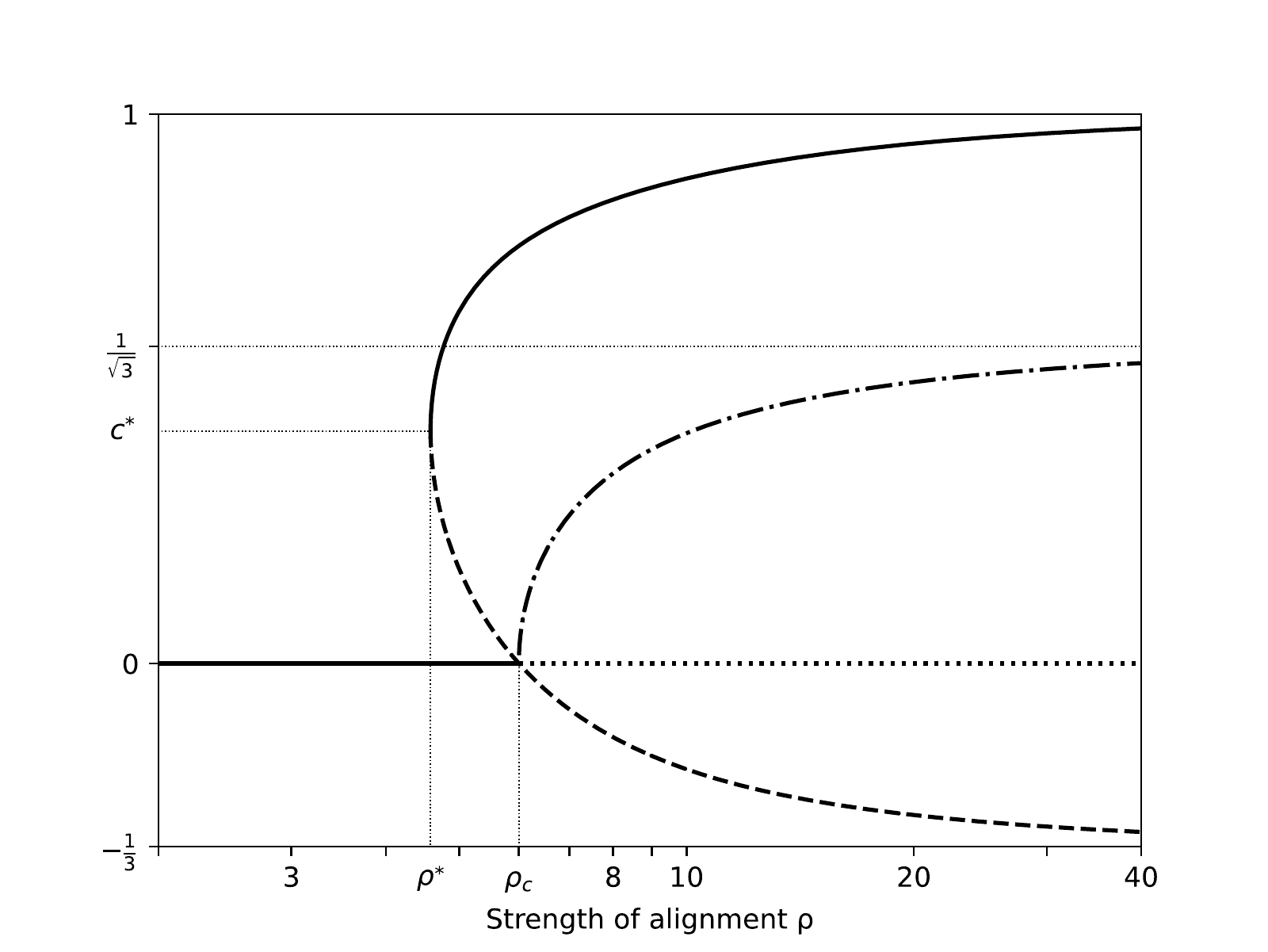}
\caption{\label{fig-plots-c}Behaviors of the functions~$\widetilde{c}_1^{↑}$ (solid line),~$\widetilde{c}_1^{↓}$ (dashed line) and~$\widetilde{c}_2$ (dashed-dot line).}
\end{center}
  
\end{figure}

We can therefore describe more precisely the long time behaviour of the solution to the Fokker--Planck equation according to the value of~$ρ$, thanks to Proposition~\ref{prop-lasalle}.

\begin{theorem} \label{thm-conv-steady} Let~$f_0$ be a nonnegative measure with mass~$ρ>0$, and~$f$ the solution to the Fokker--Planck equation~\eqref{FPJf} with initial condition~$f_0$. For the following statements, the notion of convergence is with respect to any~$H^s$ norm on~$SO_3(ℝ)$.
  \begin{itemize}
  \item If~$ρ<ρ^*$, the only steady-state is the uniform distribution on~$SO_3(ℝ)$, and the solution~$f(t,·)$ converges to this steady state as~$t→+∞$.
  \item If~$ρ^*⩽ρ⩽ρ_c$, there are three families of steady-states (two of which are equal when~$ρ=ρ^*$ or~$ρ=ρ_c$), and~$f(t,·)$ converges to one of these families:
    \begin{itemize}
      \item either there exists~$A_0(t)∈SO_3(ℝ)$ such that~$f(t,·)-ρM_{α_1^{↑}(ρ)A_0(t)}$ converges to zero,
      \item either~$f(t,·)$ converges to the uniform distribution on~$SO_3(ℝ)$,
      \item or there exists~$A_0(t)∈SO_3(ℝ)$ such that~$f(t,·)-ρM_{α_1^{↓}(ρ)A_0(t)}$ converges to zero, as~$t→+∞$.
    \end{itemize}
  \item If~$ρ>ρ_c$, there is an additional family of steady-states, and~$f(t,·)$ converges to one of these four families:
    \begin{itemize}
      \item either there exists~$A_0(t)∈SO_3(ℝ)$ such that~$f(t,·)-ρM_{α_1^{↑}(ρ)A_0(t)}$ converges to zero,
      \item either~$f(t,·)$ converges to the uniform distribution on~$SO_3(ℝ)$,
      \item either there exists~$A_0(t)∈SO_3(ℝ)$ such that~$f(t,·)-ρM_{α_1^{↓}(ρ)A_0(t)}$ converges to zero,
      \item or there exist unit vectors~$\mathbf{a}_0(t),\mathbf{b}_0(t)$ such that~$f(t,·)-ρM_{α_2^{↑}(ρ)\sqrt{3}\mathbf{a}_0(t)⊗\mathbf{b}_0(t)}$ converges to zero, as~$t→+∞$.
    \end{itemize}
  \end{itemize}
\end{theorem}
\begin{proof} This result is a summary of the possible steady-states according to Proposition~\ref{prop-equilibria} and~\ref{prop-solutions-compat}. The convergence of~$f$ to one of this families comes from Proposition~\ref{prop-lasalle} and from the fact that, even if the limit set~$\mathcal{E}_∞$ of equilibria may consist of several distinct such families, they would belong to different connected components of~$\mathcal{E}_∞$.
\end{proof}

Let us now try to understand the stability of each of these families of equilibria. Figure~\ref{fig-scatter-c} is a zoom on the region~$ρ∈[3,8]$ of the plots of the functions~$\widetilde{c}_1^{↑}$,~$|\widetilde{c}_1^{↓}|$ and~$\widetilde{c}_2$ (remember that these functions are the order parameters of the corresponding steady-states), on top of the final values of the order parameters of the numerical simulations which were given in the right part of Figure~\ref{fig-scatterplots}. It suggests the only stable equilibria, apart from the uniform one when~$ρ<ρ_c$, are those corresponding to the curve~$\widetilde{c}_1^{↑}$. This is indeed what we will show in the next section.

\begin{figure}[h!]
  \begin{center}
\includegraphics[width=15cm]{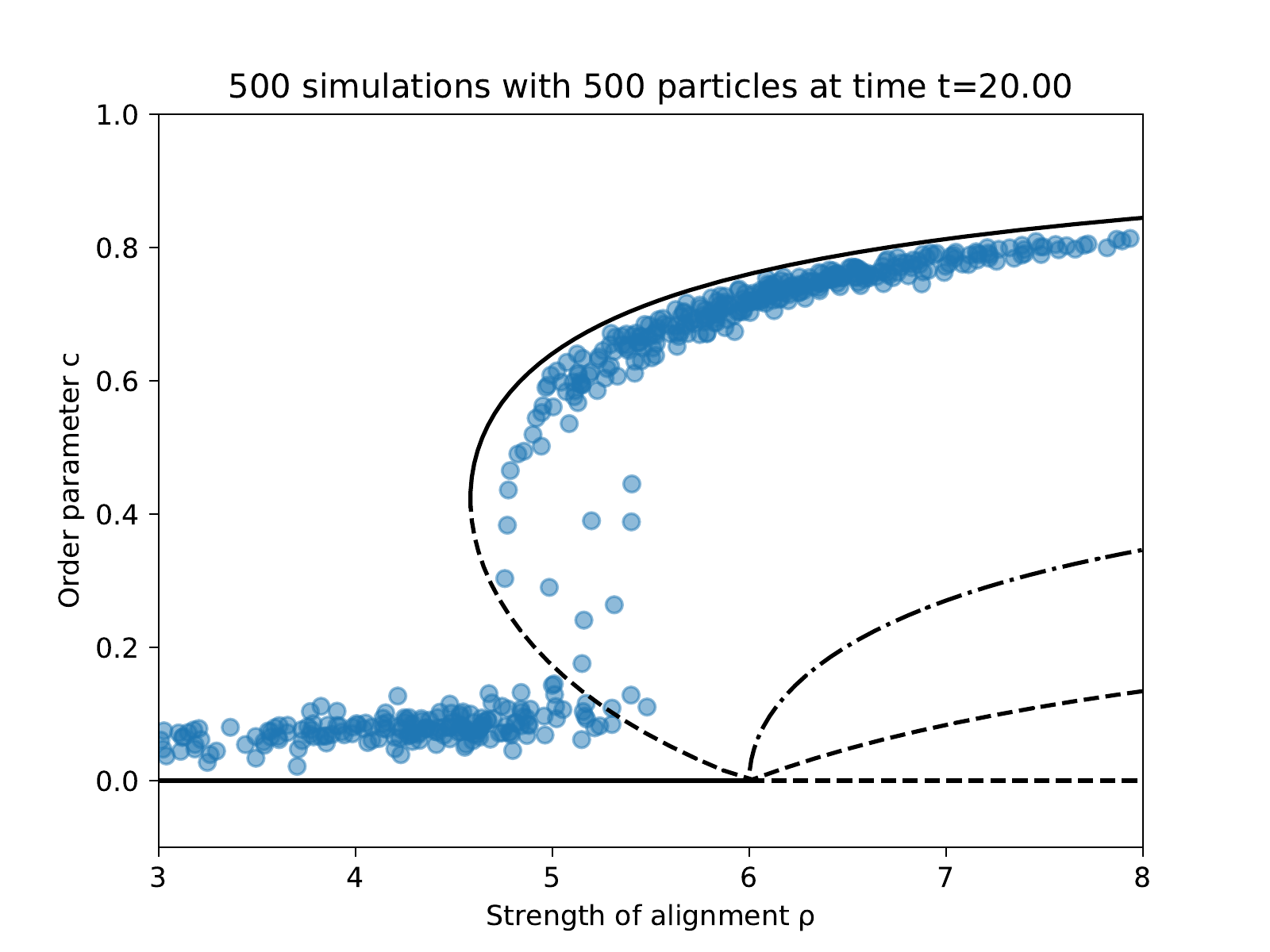}
\caption{\label{fig-scatter-c}Behaviors of the functions~$\widetilde{c}_1^{↑}$,~$|\widetilde{c}_1^{↓}|$ and~$\widetilde{c}_2$ and final order parameters of the numerical simulations.}
\end{center}
  
\end{figure}
\section{Stability results thanks to a BGK model}
\label{sec-BGK}
Instead of the Fokker--Planck equation~\eqref{FPJf}, let us consider the following BGK equation:
\begin{equation}
  \label{BGKJf} ∂_tf=ρM_{\mathcal{J}[f]}-f.
\end{equation}
This is still an equation where the total mass is preserved and for which the steady states satisfy the same compatibility equation: if~$f$ is a steady-state, it has to be of the form~$ρM_J$ where~$J=\mathcal{J}[f]=ρ\mathcal{J}[M_J]$. Therefore these two evolution equations share the same steady-states, which were determined in~\cite{degond2020phase} and summarized in the previous section. Let us now give a summary of the results of stability of these equilibria which were obtained in~\cite{degond2020phase}. It happens that these two evolution equations (BGK and Fokker--Planck) also share the same property of dissipation of the free energy~$\mathcal{F}$: if~$f$ is a positive solution to~\eqref{BGKJf}, then by multiplying both sides by~$\ln f(A)-A·\mathcal{J}[f]$ and integrating on~$SO_3(ℝ)$, we obtain
\begin{equation*}\frac{\d}{\d t}\mathcal{F}[f]+\widetilde{\mathcal{D}}[f]=0\label{dissipationBGK},
\end{equation*}
where~$\mathcal{F}[f]$ is given by~\eqref{def-F} and
\begin{equation*}
  \label{def-Dtilde}
  \widetilde{\mathcal{D}}[f]=∫_{SO_3(ℝ)}(f-ρM_{\mathcal{J}[f]})\big(\ln f - \ln(ρM_{\mathcal{J}[f]})\big)\d A⩾0.
\end{equation*}
Then, by writing~$J(t)=\mathcal{J}[f(t,·)]$ where~$f$ is a solution of the BGK equation~\eqref{BGKJf}, we obtain that~$J$ satisfies an ordinary differential equation:
\begin{equation}
  \label{ODEJBGK}
  \frac{\d}{\d t}J=ρ\mathcal{J}[M_J]-J.
\end{equation}
The long-time behaviour of the solution of the BGK equation is much simpler to study, since it can be reduced to the study of a finite dimensional ODE.

A further reduction can be done through the special singular value decomposition, for which we state a result which will be useful in the following.
\begin{proposition}\label{ssvd} If~$J∈M_3(ℝ)$, we call Special Singular Value Decomposition (SSVD) of~$J$ a decomposition of the form~$J=PDQ$ where~$D=\mathrm{diag}(d_1,d_2,d_3)$ is a diagonal matrix satisfying~$d_1⩾d_2⩾|d_3|$ and~$P,Q∈SO_3(ℝ)$.

  Such a SSVD always exists, and the matrix~$D$ is unique (the rotations~$P$ and~$Q$ may not be unique). Furthermore, we have
  \begin{equation}\label{minSSVD}
    \min_{A∈SO_3(ℝ)}∥J-A∥=∥J-PQ∥=∥D-I_3∥.
  \end{equation}
\end{proposition}
\begin{proof}
  The existence and uniqueness can be obtained through the singular value decomposition, and modifying the orthogonal matrices if necessary to change the sign of the last entry of the diagonal part and get special orthogonal matrices, see~\cite{degond2020phase}.
  We now compute
  \[∥J-A∥^2=∥D-P^{\top}AQ^{\top}∥^2=∥D∥^2-2\,B·D+\frac32,\]
  where~$B=P^{\top}AQ^{\top}$. Therefore minimizing~$∥J-A∥$ for~$A∈SO_3(ℝ)$ amounts to maximizing~$B·D$, for~$B∈SO_3(ℝ)$. The set of diagonal parts of rotation matrices (seen as vectors of~$ℝ^3$) is given by Horn’s tetrahedron~\cite{horn1954doubly}: this is the convex hull~$\mathcal{T}$ of the points~$(±1,±1,±1)$ with an even number of minus signs. Therefore we want to maximize~$\mathbf{x}·\mathbf{d}$ for~$\mathbf{x}∈\mathcal{T}$ and~$\mathbf{d}=(d_1,d_2,d_3)$. This convex function reaches it maximum on extremal points of~$\mathcal{T}$, that is to say on one of the vertices of~$\mathcal{T}$. Since we have
  \[d_1+d_2+d_3⩾d_1-d_2-d_3⩾-d_1+d_2-d_3⩾-d_1-d_2-d_3,\]
  we see that the maximum is reached for~$\mathbf{x}=(1,1,1)$. Therefore the maximum of~$B·D$ for~$B∈SO_3(ℝ)$ is reached for~$B=I_3$, which ends the proof\footnote{Let us remark that if~$d_2>-d_3$, the maximum of~$\mathbf{x}·\mathbf{d}$ is unique on~$\mathcal{T}$ and since the only rotation matrix for which the diagonal part is~$(1,1,1)$ is the identity matrix~$I_3$, we get that the minimizer~$PQ$ of~\eqref{minSSVD} is unique. So even if~$P$ and~$Q$ may not be unique, in that case the matrix~$PQ$ is unique, and could be seen as a Special Polar Decomposition of~$J$ (with the analogy with the fact that if~$\det J>0$, then~$J=PDQ$ is the singular value decomposition of~$J$ and~$PQ$ is the polar decomposition of~$J$~\cite{degond2017new}).}. 
\end{proof}

With this definition of the SSVD, the reduction that can be done is that the flow of the ODE~\eqref{ODEJBGK} preserves the SSVD: if a SSVD of the initial condition is given by~$J(0)=PD_0Q$, then for all time~$t$, we have the following SSVD:~$J(t)=PD(t)Q$, with the same rotation matrices~$P$ and~$Q$, and where~$D(t)=(d_1(t),d_2(t),d_3(t))$ is a diagonal matrix satisfying the same ODE~\eqref{ODEJBGK} as~$J$, with initial condition~$D_0$ (the fact that the matrix is diagonal and the inequalities~$d_1(t)⩾d_2(t)⩾|d_3(t)|$ are preserved by the flow of this ODE). We therefore only have to study a three-dimensional ODE.
Finally, the last observation we can do is that the flow of the ODE~\eqref{ODEJBGK} is actually a gradient flow of a potential: if we write
\begin{equation}\label{def-V}
  V(J)=\frac12∥J∥^2-ρ\ln \mathcal{Z}(J)\text{, where }\mathcal{Z}(J)=∫_{SO_3(ℝ)}\exp(J·A)\d A,
\end{equation}
as in the definition~\eqref{defVM} of the generalized von Mises distribution, we obtain
\begin{equation}\label{gradV}
  ∇V(J)=J-ρ\mathcal{J}[M_J],
\end{equation}
where the gradient is taken with respect to the inner product of~$M_3(ℝ)$ given by~\eqref{dotSO3}.

Therefore the ODE~\eqref{ODEJBGK} is simply~$\frac{\d}{\d t}J=-∇V(J)$, and one can prove that any solution will converge to a critical point of~$V$, which corresponds to a solution of the compatibility equation~\eqref{compatJ}. We then obtain the same type of convergence as in Theorem~\ref{thm-conv-steady}. The main difference is that we have convergence to a unique steady-state (and not to a set of steady-states), that can be determined by knowing a special singular value decomposition of~$\mathcal{J}[f_0]$. The other difference is that the convergence does not takes place in any Sobolev space~$H^s$: the BGK equation is not regularizing in time. The following proposition is a summary of results in~\cite{degond2020phase}:
\begin{proposition} \label{prop-conv-steady-BGK} Let~$f_0$ be a finite nonnegative measure with mass~$ρ>0$, and~$f$ the solution to the BGK equation~\eqref{BGKJf} with initial condition~$f_0$. We write the decomposition~$\mathcal{J}[f_0]=P_0D_0Q_0$, where~$P_0,Q_0∈SO_3(ℝ)$ and~$D_0=\mathrm{diag}(d_{1,0},d_{2,0},d_{3,0})$, with~$d_{1,0}⩾d_{2,0}⩾|d_{3,0}|$ (special singular value decomposition).  
  Then for all~$t∈ℝ$, we have~$\mathcal{J}[f(t,·)]=P_0D(t)Q_0$, where~$D(t)=\mathrm{diag}(d_1(t),d_2(t),d_3(t))$ is the solution to the ODE~\eqref{ODEJBGK} with initial condition~$D_0$, satisfying~$d_1(t)⩾d_2(t)⩾|d_3(t)|$. In the following statements, the notion of convergence of~$f(t,·)$ is in the space of measures (or any normed space for which~$f_0$ is an element and for which the map~$f↦\mathcal{J}[f]$ is continuous).
  \begin{itemize}
  \item If~$ρ<ρ^*$, then~$D(t)→0$ and~$f(t,·)$ converges to the uniform distribution as~$t→+∞$.
  \item If~$ρ^*⩽ρ⩽ρ_c$, there are three families of steady-states (two of which are equal when~$ρ=ρ^*$ or~$ρ=ρ_c$), and~$f(t,·)$ converges to one of these steady-states, as~$t→+∞$:
    \begin{itemize}
      \item either~$D(t)→0$, and~$f(t,·)$ converges to the uniform distribution,
      \item either~$D(t)→α_1^{↑}(ρ)I_3$, and~$f(t,·)→ρM_{α_1^{↑}(ρ)A_0}$ where~$A_0=P_0Q_0$,
      \item or~$D(t)→α_1^{↓}(ρ)I_3$, and~$f(t,·)→ρM_{α_1^{↓}(ρ)A_0}$ where~$A_0=P_0Q_0$.
    \end{itemize}
  \item If~$ρ>ρ_c$, there is an additional family of steady-state, and~$f(t,·)$ converges to one of these steady-states, as~$t→+∞$:
    \begin{itemize}
      \item either~$D(t)→0$, and~$f(t,·)$ converges to the uniform distribution,
      \item either~$D(t)→α_1^{↑}(ρ)I_3$, and~$f(t,·)→ρM_{α_1^{↑}(ρ)A_0}$ where~$A_0=P_0Q_0$,
      \item either~$D(t)→α_1^{↓}(ρ)\,\mathrm{diag}(-1,-1,1)$, and~$f(t,·)$ converges to~$ρM_{α_1^{↓}(ρ)A_0}$, where~$A_0=P_0\,\mathrm{diag}(-1,-1,1)Q_0$
      \item or~$D(t)→α_2(ρ)\mathrm{diag}(\sqrt{3},0,0)$, and~$f(t,·)$ converges to~$ρM_{α_2(ρ)\sqrt{3}\mathbf{a}_0⊗\mathbf{b}_0}$, with~$\mathbf{a}_0=P_0\mathbf{e}_1$ and~$\mathbf{b}_0=Q_0^{\top}\mathbf{e}_1$ (where~$\mathbf{e}_1$ is the first element of the canonical basis of~$ℝ^3$).
    \end{itemize}
  \end{itemize}
\end{proposition}

We now turn to stability results. For convenience, we will denote~$\overline{V}$ the restriction of~$V$ to the space of diagonal matrices. Its Hessian~$\mathrm{Hess}\,\overline{V}$ is then a symmetric bilinear form on a space of dimension~$3$. Thanks to the study of the signature of this Hessian, we obtained in~\cite{degond2020phase} the characterization of the stability of all steady-states. The next proposition is a summary of these results (without details on the domains of convergence):

\begin{proposition}\label{prop-stability-BGK}
  The uniform steady-state for the BGK equation~\eqref{BGKJf} corresponds to the critical point~$0$ of the potential~$\overline{V}$ (and~$V$).
  \begin{itemize}
  \item If~$0<ρ<ρ_c$, the Hessian~$\mathrm{Hess}\,\overline{V}(0)$ has signature~$(+++)$ (and so~$0$ is a local minimizer of~$V$). Therefore the uniform steady-state is locally asymptotically stable (with exponential rate of convergence). 
  \item If~$ρ>ρ_c$, the signature is~$(---)$ (therefore~$0$ is not a local minimizer of~$V$), and the uniform steady-state is unstable.
  \end{itemize}
  When~$ρ⩾ρ^*$, the steady-states of the form~$ρM_{α_1^{↑}(ρ)A_0}$ (resp.~$ρM_{α_1^{↓}(ρ)A_0}$) with~$A_0$ in~$SO_3(ℝ)$ (see Theorem~\ref{thm-conv-steady}) correspond to the critical points of the form~$α_1^{↑}(ρ)A_0$ (resp.~$α_1^{↓}(ρ)A_0$) of~$V$. Their nature can be reduced to the study of the critical point~$D_∞^{↑}=α_1^{↑}(ρ)I_3$  (resp.~$D_∞^{↓}=α_1^{↓}(ρ)I_3$) of~$\overline{V}$.
  \begin{itemize}
  \item If~$ρ>ρ^*$, the Hessian~$\mathrm{Hess}\,\overline{V}(D_∞^{↑})$ has signature~$(+++)$ (and so~$α_1^{↑}(ρ)A_0$ is a local minimizer of~$V$). Therefore the steady-states of the form~$ρM_{α_1^{↑}(ρ)A_0}$ are locally asymptotically stable (with exponential rate of convergence). 
  \item If~$ρ^*<ρ<ρ_c$ (resp.~$ρ>ρ_c$), the Hessian~$\mathrm{Hess}\,\overline{V}(D_∞^{↓})$ has signature~$(-++)$ (resp.~$(+--)$) (therefore~$α_1^{↓}(ρ)A_0$ is not a local minimizer of~$V$), and the steady-states of the form~$ρM_{α_1^{↑}(ρ)A_0}$ are unstable.
  \end{itemize}
  When~$ρ>ρ_c$, the steady-states of the form~$ρM_{α_2(ρ)\sqrt{3}\mathbf{a}_0⊗\mathbf{b}_0}$  with~$\mathbf{a}_0,\mathbf{b}_0∈𝕊^2$ (see Theorem~\ref{thm-conv-steady}) correspond to the critical points of the form~$α_2(ρ)\sqrt{3}\mathbf{a}_0⊗\mathbf{b}_0$ of~$V$, which reduces to the study of the critical point~$D_∞=α_2(ρ)\mathrm{diag}(1,0,0)$ of~$\overline{V}$. \begin{itemize}
    \item The Hessian~$\mathrm{Hess}\,\overline{V}(D_∞)$ has signature~$(++-)$ (therefore~$α_2(ρ)\sqrt{3}\mathbf{a}_0⊗\mathbf{b}_0$ is not a local minimizer of~$V$), and the steady-states of the form~$ρM_{α_2(ρ)\sqrt{3}\mathbf{a}_0⊗\mathbf{b}_0}$ are unstable.
    \end{itemize}
Furthermore, the critical cases are unstable: the uniform steady-state is unstable for~$ρ=ρ_c$, and the steady-states of the form~$ρM_{α^*A_0}$ are unstable when~$ρ=ρ^*$ (the corresponding matrices~$J=0$ or~$J=α^*A_0$ are not local minimizers of~$V$).
\end{proposition}

The main object of this section is to show, as it was claimed in Remark~$5.5$ of~\cite{degond2020phase}, that we can directly use these results of (in)stability for the BGK equation (and more precisely for the potential~$V$) to obtain (in)stability results for the Fokker--Planck equation, in order to complete the results around the uniform distribution given by Proposition~\ref{prop-stability-uniform}. We provide a proposition and a theorem which give details on this statement.

The first proposition allows to compare the behaviours of~$V$ and of~$J↦\mathcal{F}[ρM_J]$. 
\begin{proposition}\label{prop-W}
  Let us define for~$J∈M_3(ℝ)$
\begin{equation*}
  W(J)=\mathcal{F}[ρM_J],\label{def-W}
\end{equation*}
Then, we have that~$∇W(J)=0$ if and only if~$∇V(J)=0$, that is to say~$J$ is a solution to the compatibility equation~\eqref{compatJ}. Furthermore, if~$J$ is such a critical point, the Hessian~$\mathrm{Hess}\,W$ has the same signature as~$\mathrm{Hess}\,V$ (and more precisely, if~$\overline{W}$ is the restriction of~$W$ to the diagonal matrices, then~$\mathrm{Hess}\,\overline{W}$ and~$\mathrm{Hess}\,\overline{V}$ have the same signature).
\end{proposition}
\begin{proof}
  We first compute
  \begin{align*}
W(J)&=∫ρ (\ln ρ + A·J - \ln \mathcal{Z}(J))M_J(A)\d A - \frac{ρ^2}2∥\mathcal{J}[M_{J}]∥^2\nonumber\\
      &=ρ\ln ρ - \ln \mathcal{Z}(J) + \frac12∥J∥^2 - \frac12∥J-ρ\mathcal{J}[M_{J}]∥^2\nonumber\\
  &=V(J)-\frac12∥∇V(J)∥^2 + ρ\ln ρ\label{W-V},
\end{align*}
thanks to~\eqref{gradV}. Therefore we obtain
\begin{equation}\label{gradW}
  ∇W(J)=∇V(J)-\mathrm{Hess}\,V(J)(∇V(J)).
\end{equation}
We want to compute the Hessian of~$V$, seen as a linear mapping from~$M_3(ℝ)$ to~$M_3(ℝ)$, symmetric with respect to the inner product of~$M_3(ℝ)$. Let us take~$H$ small in~$M_3(ℝ)$. We first have that~$\mathcal{Z}(J+H)=(1+\mathcal{J}[M_J]·H)\mathcal{Z}(J)+O(∥H∥^2)$. Thus we get~$M_{J+H}(A)=(1+A·H-\mathcal{J}[M_J]·H)M_J(A)+O(∥H∥^2)$. Finally we obtain
\[\mathcal{J}[M_{J+H}]=\mathcal{J}[M_J]-(\mathcal{J}[M_J]·H)\mathcal{J}[M_J] + ∫_{SO_3(ℝ)}A (A·H)M_J(A)\d A + O(∥H∥^2).\]
Now, using the expression~\eqref{gradV} of~$∇V$, we get
\begin{equation*}
  \mathrm{Hess}\, V(J)(H)=H-ρ\Big[(\mathcal{J}[M_J]·H)\mathcal{J}[M_J]-∫_{SO_3(ℝ)}A (A·H)M_J(A)\d A\Big].\label{HessVH}
\end{equation*}
Said differently, seeing now~$\mathrm{Hess}\,V$ as a symmetric bilinear form on~$M_3(ℝ)$:
\begin{align}
  \mathrm{Hess}\, V(J)(H,H)&=∥H∥^2-ρ\Big[(\mathcal{J}[M_J]·H)^2-∫_{SO_3(ℝ)} (A·H)^2M_J(A)\d A\Big]\nonumber\\
  &=∥H∥^2-ρ∫[(A-\mathcal{J}[M_J])·H]^2M_J(A)\d A\label{HessVHH},
\end{align}
and we see that all the eigenvalues of~$\mathrm{Hess}\,V$ are strictly less than~$1$. Therefore the (symmetric) linear mapping~$\mathrm{Id}-\mathrm{Hess}\,V$ from~$M_3(ℝ)$ to~$M_3(ℝ)$ has only strictly positive eigenvalues, and is therefore an isomorphism. 
The expression~\eqref{gradW} of~$∇W$ then provides the equivalence between critical points for~$V$ and for~$W$.

Finally, at a point~$J$ for which~$∇V(J)=0$, we obtain
\begin{equation*}\label{HessWeq}
  \mathrm{Hess}\, W(J)=\mathrm{Hess}\, V(J)-[\mathrm{Hess}\, V(J)]^2.
\end{equation*}
Therefore, the eigenvalues of~$\mathrm{Hess}\, W(J)$ are given by~$λ(1-λ)$, where~$λ$ are the eigenvalues of~$\mathrm{Hess}\, W(J)$, which all satisfy~$λ<1$. Therefore their signs are the same. And this is also true when restricted to the space of diagonal matrices.
\end{proof}

We can now state the final theorem of this section.
\begin{theorem}\label{thm-stabilityFP} The nature of all the critical points of the free energy~$\mathcal{F}$ is given by the following statements.
  \begin{itemize}
    \item For~$ρ<ρ_c$, the uniform equilibrium of mass~$ρ$ is a local strict minimizer of the free energy~$\mathcal{F}$.
    \item For~$ρ>ρ^*$, the set~$\mathcal{E}=\{ρM_{α_1^{↑}(ρ)A_0},A_0∈SO_3(ℝ)\}$ is a local strict minimizer of the free energy~$\mathcal{F}$, in the sense that there exists a neighborhood~$\mathcal{V}$ of~$\mathcal{E}$ (in the space of nonnegative measures of mass~$ρ$) such that if~$f∈\mathcal{V}\setminus\mathcal{E}$, then~$\mathcal{F}[f]>\mathcal{F}_∞$, where~$\mathcal{F}_∞$ is the common value of~$\mathcal{F}$ on~$\mathcal{E}$.
    \item For~$ρ⩾ρ_c$, the uniform equilibrium of mass~$ρ$ is not a local minimizer of the free energy~$\mathcal{F}$.
    \item For~$ρ⩾ρ^*$ (and~$ρ≠ρ_c$), any steady-state of the form~$ρM_{α_1^{↓}(ρ)A_0}$ for~$A_0∈SO_3(ℝ)$ is not a local minimizer of the free energy~$\mathcal{F}$.
    \item For~$ρ>ρ_c$, any steady-state of the form~$ρM_{α_2(ρ)\sqrt3\mathbf{a}_0⊗\mathbf{b}_0}$ for~$A_0∈SO_3(ℝ)$ is not a local minimizer of the free energy~$\mathcal{F}$.
    \end{itemize}
    Therefore, the last three families of steady-states are unstable for the Fokker--Planck equation~\eqref{FPJf}: there exist initial conditions arbitrarily close to these families (in any~$H^s$ norm), such that the solution to the Fokker--Planck equation converges in long time towards another family of equilibria (see Theorem~\ref{thm-conv-steady}).
  \end{theorem}
  \begin{proof}The first point has been proven in Proposition~\ref{prop-stability-uniform}. For the second one, if it was not true, there would exist~$f_0$ as close as we want from~$\mathcal{E}$ such that~$\mathcal{F}(f_0)⩽\mathcal{F}_∞$, and~$f_0∉\mathcal{E}$. Since the different families of steady-states are isolated,~$f_0$ cannot be a steady-state. By letting~$f$ be the solution of the BGK equation with initial condition~$f_0$, we would have~$\widetilde{Q}[f_0]>0$ and therefore~$\mathcal{F}[f(t,·)]<\mathcal{F}_∞$ for all~$t>0$. Combined with the fact that~$\mathcal{F}[f(t,·)]$ is nonincreasing in time, this would be in contradiction with the fact that~$f(t,·)$ converges towards the set~$\mathcal{E}$, thanks to the asymptotic stability of those steady-states for the BGK equation given by Proposition~\ref{prop-stability-BGK}.
    Let us remark that the first point of the theorem could be proven in the same way, without having to expand the free energy, but only using the known results for the BGK equation and the fact that~$\mathcal{F}$ is nonincreasing.

    To prove the last three points, let us take such a steady state, of the form~$ρM_{J_0}$. We want to prove that~$J_0$ is not a local minimizer of~$W$, therefore~$ρM_{J_0}$ is not a local minimizer of~$\mathcal{F}$. We write a SSVD of the form~$J_0=PD_0Q$ where~$D_0$ is a diagonal matrix and~$P,Q∈SO_3(ℝ)$. If~$J=PDQ$ where~$D$ is a diagonal matrix close to~$D_0$, then~$W(J)=\overline{W}(D)$. Therefore we only need to prove that~$D_0$ is not a local minimizer of~$\overline{W}$.
    In the case where~$ρ≠ρ_c$ and~$ρ≠ρ^*$, since the signature of~$\mathrm{Hess}\,\overline{W}(D_0)$ has negative components (thanks to Propositions~\ref{prop-stability-BGK} and~\ref{prop-W}), we directly get the results.
    In the critical cases we will use a mountain-pass lemma argument. In the case where~$ρ=ρ_c$, suppose that~$0$ is a local minimizer of~$\overline{W}$. Then it is a local strict minimizer, since this critical point is isolated. Therefore by looking at the other local strict minimizer~$α_1^{↑}(ρ_c)I_3$ of~$\overline{W}$ (for which the signature of the Hessian is~$(+++)$, thanks again to Propositions~\ref{prop-stability-BGK} and~\ref{prop-W}), we would obtain, by the mountain-pass lemma, a third critical point~$D$ of~$\overline{W}$, which would satisfy~$\overline{W}(D)>\max(\overline{W}(0),\overline{W}(α_1^{↑}(ρ_c)I_3))$. This is in contradiction with the fact that we only have two families of equilibria for this value of~$ρ$.
    The same argument can be used to show that when~$ρ=ρ^*$, the point~$α^*I_3$ is not a local minimizer of~$\overline{W}$, using as other local strict minimizer the point~$0$.

    The conclusion of the statement of the theorem comes from the fact that we actually proved that the critical points were not local minimizers of~$W$, which is the evaluation of~$\mathcal{F}$ on smooth functions of the form~$ρM_J$, so the~$H^s$ norm of~$ρM_{J}-ρM_{J_0}$ is small when~$J$ is close to~$J_0$. 
  \end{proof}

  For the first two points of Theorem~\ref{thm-stabilityFP}, we did not provide the corresponding stability results. Indeed, in the next section, a more detailed study will show that they are exponentially stable.
  \section{Exponential convergence for the stable steady-states}
  \label{sec-exponential}
We will now show that the two families of steady-states that correspond to what we observe in the numerical simulations are locally exponentially attracting. In particular, when~$f$ is a solution to the Fokker--Planck equation in the neighborhood of those steady-states, we will show that~$\mathcal{J}[f(t,·)]$ will converge to a solution~$J_∞$ of the compatibility equation~\eqref{compatJ}. However, since this~$J_∞$ (if it is non-zero) is not known from the initial condition (contrary to the case of the BGK equation), it is not easy to control directly the distance between~$f$ and~$ρM_{J_∞}$, but we will see that controlling the distance from~$f$ and~$ρM_{\mathcal{J}[f]}$, even if this last one is not a steady-state, will be the key to our analysis. A convenient framework is to use the relative entropy, for which we will need the following results.

\begin{proposition}\label{prop-relative-entropy} Let~$ρ>0$. If~$f,g$ are two measurable nonnegative functions on~$SO_3(ℝ)$ with total mass~$ρ$ and with~$g>0$, we define the relative entropy and Fisher information by
\begin{equation*}
  \label{relativeEntropyFisher}
  \mathcal{H}(f|g)=∫_{SO_3(ℝ)}\!\!f(A)\ln\Big(\frac{f(A)}{g(A)}\Big)\d A, \quad \mathcal{I}(f|g)=∫_{SO_3(ℝ)}\!\!f(A)\Big{∥}∇\ln\Big(\frac{f(A)}{g(A)}\Big)\Big{∥}^2\d A.
\end{equation*}
Then, for two such functions, we have the Csiszár--Kullback--Pinsker inequality:
\begin{equation}
  \label{pinsker}
  ∫_{SO_3(ℝ)}|f(A)-g(A)|\,\d A⩽\sqrt{2ρ\, \mathcal{H}(f|g)}.
\end{equation}
Finally, we have the following families of (weighted) logarithmic Sobolev inequalities: there exists a constant~$λ>0$ such that for all~$J∈M_3(ℝ)$ with~$∥J∥⩽\frac{\sqrt{3}}{\sqrt2}ρ$, and all measurable nonnegative function~$f$ with total mass~$ρ$, we have
\begin{equation}\label{logsob}
  \mathcal{H}(f|ρM_J)⩽\frac1{2λ}\,\mathcal{I}(f|ρM_J).
\end{equation}
\end{proposition}
\begin{proof} The Csiszár--Kullback--Pinsker inequality is well-known~\cite{csiszar1967information,pinsker1964information}, we just notice the factor~$ρ$ since we do not work with probability measures here. The logarithmic Sobolev inequality~\eqref{logsob} in the case~$J=0$ (uniform measure on~$SO_3(ℝ)$) come for instance from the Bakry--Émery criterion~\cite{bakry1985diffusions} since~$SO_3(ℝ)$ has positive Ricci curvature (this is the same as the curvature of~$𝕊^3$, thanks to the local isometry~$Φ$ given in Proposition~\ref{prop-quaternions-rotations})\footnote{Actually, as already stated by Bakry and Émery~\cite{bakry1985diffusions}, this criterion does not give the optimal constant in~$𝕊^3$, which was given by Mueller and Weissler in~\cite{mueller1982hypercontractivity}, but here even the optimal constant in~$𝕊^3$ would not be necessarily optimal in~$SO_3(ℝ)$, since we only want the logarithmic Sobolev inequality for even functions on~$𝕊^3$.}. Then, we use the fact that the logarithmic Sobolev inequality is stable by bounded perturbation~\cite{holley1987logarithmic,villani2003topics}. Since~$∥J∥⩽\frac{\sqrt{3}}{\sqrt2}ρ$, then~$M_J$ is bounded above and below, uniformly in~$J$, which ends the proof. 
\end{proof}

Let us now compute the relative entropy of~$f$ with respect to~$ρM_{J}$ for~$J∈M_3(ℝ)$. Using the definition~\eqref{defVM}, we obtain
\begin{align}
  \mathcal{H}(f|ρM_{J})&=∫_{SO_3(ℝ)}\big(f(A)\ln f(A)-f(A)\,A·J\big)\d A + ρ\ln\mathcal{Z}(J) - ρ\ln ρ\nonumber\\
  &=\mathcal{F}[f]+\frac12∥J-\mathcal{J}[f]∥^2-V(J)-ρ\ln ρ,\label{HFJV}
\end{align}
thanks to the definitions~\eqref{def-F} and~\eqref{def-V} of~$\mathcal{F}$ and~$V$. Therefore, if~$J_{\mathrm{eq}}$ is a solution to the compatibility equation and~$f_{\mathrm{eq}}=ρM_{J_{\mathrm{eq}}}$, we apply~\eqref{HFJV} with~$f=f_{\mathrm{eq}}$ and~$J=J_{\mathrm{eq}}$ to obtain~$ρ\ln ρ=\mathcal{F}[f_{\mathrm{eq}}]-V(J_{\mathrm{eq}})$. Now applying~\eqref{HFJV} with~$J=J_{\mathrm{eq}}$ or with~$J=\mathcal{J}[f]$, we obtain 
\begin{gather}\label{FHV}
  \mathcal{F}[f]-\mathcal{F}[f_{\mathrm{eq}}]=\mathcal{H}(f|ρM_{\mathcal{J}[f]})+V(\mathcal{J}[f])-V(J_{\mathrm{eq}}),\\
  \mathcal{H}(f|f_{\mathrm{eq}})=\mathcal{F}[f]-\mathcal{F}[f_{\mathrm{eq}}]+\frac12∥J_{\mathrm{eq}}-\mathcal{J}[f]∥^2.\label{HFJ}
\end{gather}
Furthermore, it is straightforward to see, thanks to the definition~\eqref{def-D} of~$\mathcal{D}[f]$, that
\begin{equation}\label{DI}
  \mathcal{D}[f]=\mathcal{I}(f|ρM_{\mathcal{J}[f]}).
\end{equation}
These links between the free energy, its dissipation, the relative entropy, the Fisher information, and the potential~$V$ associated to the BGK equation are the key points to prove the stability of the steady-states associated to solutions of the compatibility equation corresponding to local minimizers of~$V$.

\begin{theorem}\label{thm-expstability-FP}
  Let~$ρ>ρ^*$ (resp.~$ρ<ρ_c$).

  We define the set of equilibria~$\mathcal{E}_∞=\{ρM_{α^{↑}_1(ρ)A_0}, A_0∈SO_3(ℝ)\}$ (resp.~$\mathcal{E}_∞$ reduced to the uniform distribution on~$SO_3(ℝ)$ of mass~$ρ$).

  Then there exists~$δ>0$,~$\widetilde{λ}>0$ and~$C>0$ such that for all nonnegative measurable function~$f_0$ with mass~$ρ$, if there exists~$f_{\mathrm{eq},0}∈\mathcal{E}_∞$ such that~$\mathcal{H}(f_0|f_{\mathrm{eq},0})<δ$, then there exists~$f_∞∈\mathcal{E}_∞$ such that for all time~$t⩾0$, we have
  \begin{equation*}
    \mathcal{H}(f(t,·)|f_∞)⩽C\,e^{-2\widetilde{λ}t}\,\mathcal{H}(f_0|f_{\mathrm{eq},0}).
  \end{equation*}
\end{theorem}

\begin{proof}
  For convenience, we write~$α=α^{↑}_1(ρ)$ (resp.~$α=0$ for the study of stability of the uniform equilibrium) and~$V_∞=V(αI_3)$. We also denote by~$E_∞$ the set of matrices~$J_{\mathrm{eq}}$ solutions to the compatibility equation~\eqref{compatJ} corresponding to the family of equilibria we are interested in, that is to say~$E_∞=\{αA_0,A_0∈SO_3(ℝ)\}$.

Since the signature of~$\mathrm{Hess}\overline{V}(αI_3)$ is~$(+++)$ (thanks to Proposition~\ref{prop-stability-BGK}), by continuity of~$\mathrm{Hess}\,\overline{V}$ (and of its smallest eigenvalue), there exists~$δ_0>0$ and~$η>0$ such that for all diagonal matrix~$D$ with~$∥D-αI_3∥<δ_0$,~$\mathrm{Hess}\,\overline{V}(D)$ is positive definite with lowest eigenvalue being greater than or equal to~$η$ (we recall that thanks to~\eqref{HessVHH}, its highest eigenvalue is always less than~$1$). By the following Taylor formulas, for all such~$D$, we have
  \begin{gather*} ∥∇\overline{V}(D)∥^2=(D-αI_3)·\Big(∫_0^1\mathrm{Hess}\,\overline{V}(αI_3+t(D-αI_3))\d t\Big)^2(D-αI_3),\\
    V(D)-V_∞=∫_0^1(1-t)(D-αI_3)·\mathrm{Hess}\,\overline{V}(αI_3+t(D-αI_3))(D-αI_3)\d t  \end{gather*}
  and therefore
  \begin{gather*}∥∇\overline{V}(D)∥^2⩾η\,∥D-αI_3∥^2,\\
    \frac{η}2\,∥D-αI_3∥^2⩽V(D)-V_∞⩽\frac12∥D-αI_3∥^2⩽\frac1{2η}∥∇\overline{V}(D)∥^2.
\end{gather*}
Therefore, we write~$U=\{J∈M_3(ℝ), \min_{J_{\mathrm{eq}∈E_∞}}∥J-J_{\mathrm{eq}}∥<δ_0\}$, which is a neighborhood of~$E_∞$. If~$J∈U$ and we write the SSVD~$J=PDQ$, we obtain by Proposition~\ref{ssvd} that~$\min_{J_{\mathrm{eq}∈E_∞}}∥J-J_{\mathrm{eq}}∥=∥D-αI_3∥⩽δ_0$ (when~$α>0$, and the result is still true if~$α=0$ since~$E_∞=\{0\}$ in that case). Therefore, since~$V(J)=V(D)$ we obtain that there exists~$J_{\mathrm{eq}}∈E_∞$ (which is equal to~$αPQ$) such that 
\begin{equation}
  \frac{η}2\,∥J-J_{\mathrm{eq}}∥^2⩽V(J)-V_∞⩽\frac1{2η}∥∇V(J)∥^2=\frac1{2η}∥J-ρ\mathcal{J}[M_J]∥^2.\label{VgradVJ}
\end{equation}
By the Csiszár--Kullback--Pinsker inequality~\eqref{pinsker}, we have that if~$g$ is a nonnegative measure with mass~$ρ$:

\begin{equation}\label{pinskerJ}
  ∥\mathcal{J}[f]-\mathcal{J}[g]∥⩽∫_{SO_3(ℝ)}∥A∥|f(A)-g(A)|\d A⩽\frac{\sqrt3}{\sqrt2}\sqrt{2ρ\mathcal{H}(f|g)},
\end{equation}
and therefore for~$g=ρM_{\mathcal{J}[f]}$, we obtain
\[∥\mathcal{J}[f]-\mathcal{J}[ρM_{\mathcal{J}[f]}]∥⩽\sqrt{3ρ\mathcal{H}(f|ρM_{\mathcal{J}[f]})}.\]
Combining this with~\eqref{FHV} and~\eqref{VgradVJ} with~$J=\mathcal{J}[f]$, we get that if~$\mathcal{J}[f]∈U$, then
\begin{equation*}
  \mathcal{F}[f]-\mathcal{F}_∞⩽(1+\tfrac{3ρ}{2η})\,\mathcal{H}(f|ρM_{\mathcal{J}[f]}).\label{FH}
\end{equation*}

Therefore, as soon as~$\mathcal{J}[f]∈U$, we have by~\eqref{DI} and the logarithmic Sobolev inequality~\eqref{logsob} (we recall that~$∥\mathcal{J}[f]∥⩽\frac{\sqrt3}{\sqrt2}ρ$ if the total mass of~$f$ is~$ρ$):
\begin{equation*}
  \mathcal{D}[f]⩾\frac{2λ}{1+\tfrac{3ρ}{2η}}(\mathcal{F}[f]-\mathcal{F}_∞).
\end{equation*}
By the dissipation of the free energy~\eqref{dissipation}, writing~$\widetilde{λ}=\frac{λ}{1+\tfrac{3ρ}{2η}}$ we obtain that as long as~$\mathcal{J}[f]∈U$,
\begin{equation}\label{Fexp}
  0⩽\mathcal{F}[f]-\mathcal{F}_∞⩽e^{-2\widetilde{λ}t}(\mathcal{F}[f_0]-\mathcal{F}_∞)⩽e^{-2\widetilde{λ}t}\mathcal{H}(f_0|f_{\mathrm{eq},0}),
\end{equation}
the first inequality coming from~\eqref{FHV} and the fact that~$V(\mathcal{J}[f])-V_∞⩾0$ thanks to~\eqref{VgradVJ}, and the last inequality coming from~\eqref{HFJ}. Finally, thanks to~\eqref{VgradVJ},~\eqref{FHV} and~\eqref{Fexp}, we obtain that still as long as~$\mathcal{J}[f(t,·)]∈U$, there exists~$J_{\mathrm{eq}}(t)$ such that
\begin{equation}\label{Jexp}
  ∥\mathcal{J}[f(t,·)]-J_{\mathrm{eq}}(t)∥⩽\sqrt{\frac2{η}(V(\mathcal{J}[f(t)])-V_∞)}⩽\frac{\sqrt{2}}{\sqrt{η}}\,e^{-\widetilde{λ}t}\sqrt{\mathcal{H}(f_0|f_{\mathrm{eq},0})}.
\end{equation}
Therefore, by taking~$δ=\min(\frac{η}2 δ_0^2,\frac1{3ρ}δ_0^2)$, and using~\eqref{pinskerJ} with~$g=f_{\mathrm{eq},0}$, we obtain that if~$\mathcal{H}(f_0|f_{\mathrm{eq},0})<δ$, then~$∥\mathcal{J}[f_0]-\mathcal{J}[f_{\mathrm{eq},0}]∥<δ_0$, so~$\mathcal{J}[f_0]∈U$, and for all positive time~$∥\mathcal{J}[f(t,·)]-J_{\mathrm{eq}}(t)∥<δ_0$ (and therefore~$\mathcal{J}[f(t,·)]$ stays in~$U$) thanks to~\eqref{Jexp}. Indeed, if it was not the case, for the first exit time~$t_0>0$ of~$U$, we would have~$∥\mathcal{J}[f(t_0,·)]-J_{\mathrm{eq}}(t_0)∥⩽\frac{\sqrt{2}}{\sqrt{η}}\sqrt{\mathcal{H}(f_0|f_{\mathrm{eq},0})}<δ_0$ which is a contradiction. From now on we suppose that~$\mathcal{H}(f_0|f_{\mathrm{eq},0})<δ$, so that~\eqref{Jexp} and~\eqref{Fexp} are valid for all time~$t⩾0$.

Let us now find a way to control the displacement of~$\mathcal{J}[f]$. For~$J∈M_3(ℝ)$, using the Fokker--Planck equation~\eqref{FPJf} and integrating by parts, we have
\begin{equation*}\label{ddtJfJ}
  \frac{\d}{\d t}{J·\mathcal{J}[f]}=∫_{SO_3(ℝ)}[∇_A(A·J)·∇_A(A·\mathcal{J}[f])-Δ_A(A·J)]f(A)\,\d A,
\end{equation*}
which can be written
\begin{equation}\label{ddtJML}
  \frac{\d}{\d t}\mathcal{J}[f]=\mathcal{M}[f](\mathcal{J}[f])-\mathcal{L}[f],
\end{equation}
where, when~$g$ is an integrable function on~$SO_3(ℝ)$, we define~$\mathcal{M}[g]$ as the linear operator from~$M_3(ℝ)$ to~$M_3(ℝ)$ given by the fact that for any~$J,J'∈M_3(ℝ)$,
\[J·\mathcal{M}[g](J')=∫_{SO_3(ℝ)}∇_A(A·J)·∇_A(A·J')g(A)\,\d A,\]
and~$\mathcal{L}[g]$ as the matrix\footnote{We can actually show (but we do not need it here) that~$\mathcal{L}[f]$ is proportional to~$\mathcal{J}[f]$. Indeed, since~$q↦q·Qq$ is an eigenfunction of the Laplacian on the unit sphere of~$ℝ^4$ (more precisely a spherical harmonic of degree~$2$) when~$Q$ is a symmetric trace-free matrix, we get, thanks to the local isometry~$Φ$ and Proposition~\eqref{isomorphism-Qtensors}, that~$A↦A·J$ is also an eigenfunction of the Laplacian on~$SO_3(ℝ)$.} such that for all~$J∈M_3(ℝ)$,
\[J·\mathcal{L}[g]=∫_{SO_3(ℝ)}Δ_A(A·J)g(A)\,\d A.\]
We therefore see that since the functions under the integral are smooth and bounded, there exists~$C_0>0$ such that for all~$J∈M_3(ℝ)$ and for any integrable function~$g$ on~$SO_3(ℝ)$,
\begin{equation}\label{boundsLM}
  ∥\mathcal{L}[g]∥⩽C_0∫_{SO_3(ℝ)}|g(A)|\,\d A\quad \text{and}\quad ∥\mathcal{M}[g](J)∥⩽C_0∥J∥∫_{SO_3(ℝ)}|g(A)|\,\d A.
\end{equation}
Therefore, defining~$f_{\mathrm{eq}}(t,·)=ρM_{J_{\mathrm{eq}}(t)}$, and using the fact that it is a stationary solution, thus giving by~\eqref{ddtJML} that~$\mathcal{M}[f_{\mathrm{eq}}](\mathcal{J}[f_{\mathrm{eq}}])-\mathcal{L}[f_{\mathrm{eq}}]=0$, we obtain
\begin{align*}
  \Big{∥}\frac{\d}{\d t}\mathcal{J}[f]\Big{∥}&=∥\mathcal{M}[f](\mathcal{J}[f])-\mathcal{M}[f_{\mathrm{eq}}](\mathcal{J}[f_{\mathrm{eq}}])-\mathcal{L}[f]+\mathcal{L}[f_{\mathrm{eq}}]∥\\
  &⩽∥\mathcal{M}[f](\mathcal{J}[f-f_{\mathrm{eq}}])∥+∥\mathcal{M}[f-f_{\mathrm{eq}}](\mathcal{J}[f_{\mathrm{eq}}])∥+∥\mathcal{L}[f-f_{\mathrm{eq}}]∥.
\end{align*}
Therefore, by using~\eqref{boundsLM} and the Csiszár--Kullback--Pinsker inequalities~\eqref{pinsker} and~\eqref{pinskerJ}, we get that there exists a constant~$C_1>0$ (only depending on~$ρ$) such that
\begin{equation*}
\Big{∥}\frac{\d}{\d t}\mathcal{J}[f]\Big{∥}⩽\sqrt{C_1\mathcal{H}(f|f_{\mathrm{eq}})}.
\end{equation*}
Combining this with~\eqref{HFJ},~\eqref{Jexp}, and~\eqref{Fexp}, we then get that there exists a constant~$C_2$ (not depending on~$f_0$) such that for all~$t⩾0$
\[\Big{∥}\frac{\d}{\d t}\mathcal{J}[f]\Big{∥}⩽e^{-\widetilde{λ}t}\sqrt{C_2\mathcal{H}(f_0|f_{\mathrm{eq},0})}.\]
Finally, this gives that~$\mathcal{J}[f]$ converges exponentially fast with rate~$\widetilde{λ}$ towards a given matrix~$J_∞∈M_3(ℝ)$ and since the distance between~$\mathcal{J}[f]$ and~$E_∞$ converges to~$0$ thanks to~\eqref{Jexp}, we obtain that~$J_∞∈E_∞$. More precisely, we have
\begin{equation}\label{JJinf}
  ∥\mathcal{J}[f(t,·)]-J_∞∥⩽∫_t^{+∞}\Big{∥}\frac{\d}{\d s}\mathcal{J}[f(s,·)]\Big{∥}\d t⩽\frac{e^{-\widetilde{λ}t}}{\widetilde{λ}}\sqrt{C_2\mathcal{H}(f_0|f_{\mathrm{eq},0})}.
\end{equation}
Defining~$f_∞=ρM_{J_∞}$ and using~\eqref{HFJ} with~$f_{\mathrm{eq}}=f_∞$, \eqref{Fexp}, and~\eqref{JJinf}, we then get that there exists a constant~$C_3>0$ (not depending on~$f_0$) such that
 \[\mathcal{H}(f(t,·)|f_∞)⩽C_3e^{-2\widetilde{λ}t}\,\mathcal{H}(f_0|f_{\mathrm{eq},0}),\]
which ends the proof. Let us remark that this proof covers the case~$α=0$, but if we only want to do this case, it can be simplified a lot since~$E_∞=\{0\}$.
\end{proof}

Let us finish this section by some comments. The proof of Theorem~\ref{thm-expstability-FP} has been done here in relative entropy. It may look similar in some points to~\cite{figalli2018global}, but the main idea is above all based on the fact that we measure the relative entropy with respect to a target measure~$ρM_{\mathcal{J}[f]}$ which is not itself a steady-state. The fine control of the potential~$V$ around the solutions of the compatibility equation is the key to link all these different quantities. The proof would have worked the same in~$L^2$, by using the regularizing effect of the equation (and~$L^∞$ bounds), as was done in~\cite{frouvelle2012dynamics} for the Vicsek model, but the main difference is again that we would compare~$\mathcal{D}[f]$ and~$\mathcal{F}[f]-\mathcal{F}_∞$ with~$∥f-ρM_{\mathcal{J}[f]}∥_2^2$. This proof seems to be adaptable to a lot of different models of Fokker--Planck type, such as the Doi--Onsager theory for suspensions of rodlike polymers, for which, as far as we know, no proof of exponential convergence is available (but the analog to the potential~$V$ has been studied, therefore the nature of the critical points is well-known). This is left for future work.

Finally, now that we have a good understanding of the long time behaviour of the Fokker--Planck equation~\eqref{FPJf}, we could try to further understand the limit of the particle system as~$N→∞$. Since the mean-field limit is essentially a law of large numbers, we expect fluctuations of order~$\frac1{\sqrt N}$, which explains why the order parameters of the numerical simulations in Figure~\ref{fig-scatter-c} are not so close to~$0$ for what is expected to be the uniform distribution. More precisely, as indicated by the estimate~\eqref{couplingEstimates}, the distance between the empirical measure and the solution to the Fokker--Planck equation can be bounded by~$\frac{e^{\widetilde{C}T}}{\sqrt{N}}$, for all~$t$ in~$[0,T]$. Therefore if we want such an estimate for a large time~$T$, we cannot do better than~$T$ of order~$\ln N$. However, since the equilibria are exponentially stable, the fluctuations that would push the empirical distribution away from the family of stable equilibria, are compensated by the deterministic dynamics of the Fokker--Planck equation. Therefore the only remaining fluctuations would cause the solution to fluctuate mainly in the tangential component of the family of equilibria. This approach has been made rigorous in the case of identical Kuramoto oscillators in~\cite{bertini2014synchronization} (which corresponds to the Vicsek model studied in~\cite{frouvelle2012dynamics} in dimension two), where it is proved that the solution stays close to the set of equilibria up to times of order~$N$, but with the center of synchronization of the distribution performing a Brownian motion on the circle at these time scales. In analogy with this result, we could expect in our case that, close to the family of von Mises distributions~$ρM_{αA}$ with~$α>0$ and~$A∈SO_3(ℝ)$, the long time behaviour at time~$t=sN$ of the empirical measure of particle system would be close to~$ρM_{αA(s)}$, where~$A(s)$ performs a Brownian motion on~$SO_3(ℝ)$. This is also left for future work.

\section*{Acknowledgements}
The author wants to thank his collaborators Pierre Degond, Sara Merino-Aceituno, Ariane Trescases and Antoine Diez for all the work done together on body-attitude coordination models~\cite{degond2017new,degond2018quaternions,degond2019alignment,degond2020phase,diez2020propagation}, which inspired the talk given at {[\rotatebox[origin=c]{180}{!}$ℕδ$A\,\raisebox{0.3ex}{\rotatebox[origin=c]{90}{\scalebox{0.8}{\reflectbox{$\sum$}}}}\,]} in November 2019, and finally led to the present paper. This work has been supported by the Project EFI ANR-17-CE40-0030 of the French National Research Agency.

\bibliographystyle{alpha}
\bibliography{../biblio.bib}

\end{document}